\documentclass[11pt, a4paper]{amsart}
\pdfoutput=1
\usepackage{tikz,enumitem,rotating,latexsym,bm,stmaryrd}
\usetikzlibrary{positioning,intersections,shapes.geometric,calc,decorations.pathreplacing}
\usepackage{graphicx}
\usepackage{amsmath,amsthm,amsfonts,amssymb,mathrsfs,pb-diagram}
\usepackage[pagebackref=true,bookmarks=true,colorlinks=true,linktoc=page,citecolor=darkgreen,linkcolor=blue,urlcolor=mediumblue]{hyperref}
\usepackage{caption}
\usepackage{lipsum,wasysym}
\usepackage{mathtools}
\usepackage[a4paper,margin = 0.9in]{geometry}

\usepackage{color}
\definecolor{mediumblue}{rgb}{0.0, 0.0, 0.8}
\colorlet{darkgreen}{green!50!black}

\renewcommand*{\backref}[1]{}
\renewcommand*{\backrefalt}[4]{%
 \ifcase #1 No citations.
 \or [Page #2.]
 \else [Pages #2.]
 \fi%
}
\usepackage{cleveref}

\crefname{defn}{Definition}{Definitions}
\crefname{thm}{Theorem}{Theorems}
\crefname{prop}{Proposition}{Propositions}
\crefname{lem}{Lemma}{Lemmas}
\crefname{cor}{Corollary}{Corollaries}
\crefname{conj}{Conjecture}{Conjectures}
\crefname{section}{Section}{Sections}
\crefname{subsection}{Subsection}{Subsections}
\crefname{eg}{Example}{Examples}
\crefname{figure}{Figure}{Figures}
\crefname{rem}{Remark}{Remarks}
\crefname{rmk}{Remark}{Remarks}
\crefname{equation}{equation}{equations}

\Crefname{defn}{Definition}{Definitions}
\Crefname{thm}{Theorem}{Theorems}
\Crefname{prop}{Proposition}{Propositions}
\Crefname{lem}{Lemma}{Lemmas}
\Crefname{cor}{Corollary}{Corollaries}
\Crefname{conj}{Conjecture}{Conjectures}
\Crefname{section}{Section}{Sections}
\Crefname{subsection}{Subsection}{Subsections}
\Crefname{eg}{Example}{Examples}
\Crefname{figure}{Figure}{Figures}
\Crefname{rem}{Remark}{Remarks}
\Crefname{rmk}{Remark}{Remarks}

\renewcommand{\geq}{\geqslant}
\renewcommand{\leq}{\leqslant}
\renewcommand{\trianglerighteq}{\trianglerighteqslant}

\tikzset{wei/.style=
{red,double=red,double
distance=0.5pt}}

\tikzset{wei2/.style={red,double=red,double
distance=0.5pt}}
 
\allowdisplaybreaks
\numberwithin{equation}{section}
\usepackage{scalefnt}
\newtheorem*{Main}{Main Theorem}
\newtheorem*{Main2}{Corollary}
\newtheorem{thm}{Theorem}[section]
\newtheorem{cor}[thm]{Corollary}

\newtheorem{lem}[thm]{Lemma}
\newtheorem{prop}[thm]{Proposition}
\newtheorem*{prop*}{Proposition}
\newtheorem*{thm*}{Theorem}

\newtheorem*{cor*}{Corollary}
\newtheorem*{conj*}{Conjecture}

\newtheorem*{move1'}{Move $\mathbf{1^*}$}

\theoremstyle{remark}

\newtheorem{rmk}[thm]{Remark}

\newtheorem{rem}[thm]{Remark}

\newtheorem*{ack}{Acknowledgements}

\theoremstyle{definition}
\newtheorem{defn}[thm]{Definition}
\newtheorem{eg}[thm]{Example}

\crefname{defn}{Definition}{Definitions}
\crefname{thm}{Theorem}{Theorems}
\crefname{prop}{Proposition}{Propositions}
\crefname{lem}{Lemma}{Lemmas}
\crefname{cor}{Corollary}{Corollaries}
\crefname{conj}{Conjecture}{Conjectures}
\crefname{section}{Section}{Sections}
\crefname{subsection}{Subsection}{Subsections}
\crefname{eg}{Example}{Examples}
\crefname{figure}{Figure}{Figures}
\crefname{rem}{Remark}{Remarks}
\crefname{rmk}{Remark}{Remarks}

\Crefname{defn}{Definition}{Definitions}
\Crefname{thm}{Theorem}{Theorems}
\Crefname{prop}{Proposition}{Propositions}
\Crefname{lem}{Lemma}{Lemmas}
\Crefname{cor}{Corollary}{Corollaries}
\Crefname{conj}{Conjecture}{Conjectures}
\Crefname{section}{Section}{Sections}
\Crefname{subsection}{Subsection}{Subsections}
\Crefname{eg}{Example}{Examples}
\Crefname{figure}{Figure}{Figures}
\Crefname{rem}{Remark}{Remarks}
\Crefname{rmk}{Remark}{Remarks}

\newcommand{\degr}{\mathrm{deg}}

\newcommand{\SStd}{\operatorname{SStd}}

\newcommand{\TSStd}{\operatorname{\mathcal{T}}}

 \newcommand{\SSTS}{\mathsf{S}}  
\newcommand{\SSTT}{\mathsf{T}}  
\newcommand{\SSTU}{\mathsf{U}}  
\newcommand{\SSTV}{\mathsf{V}}  
\newcommand{\SSTQ}{\mathsf{Q}}  
\newcommand{\SSTR}{\mathsf{R}}  
\newcommand{\ZZ}{{\mathbb Z}}

\newcommand{\g}{\ell}
\newcommand{\CC}{{\mathbb C}}
\newcommand{\RR}{{\mathbb R}}
\newcommand{\la}{{\lambda}}
\newcommand\mptn[2]{\mathscr{P}^{#1}_{#2}}

\newcommand\bbz{\mathbb{Z}}

\newcommand{\doms}{\vartriangleright}
\newcommand{\ndom}{\ntrianglerighteqslant}

\newcommand{\domby}{\trianglelefteqslant}

\DeclareMathOperator{\Ext}{Ext} 
\DeclareMathOperator{\Hom}{Hom}

\let\<=\langle
\let\>=\rangle

\newcommand\Dec[1][A]{\mathbf{D}_{#1}(t)}

\tikzset{
    ultra thin/.style= {line width=0.05pt},
    very thin/.style=  {line width=0.2pt},
    thin/.style=       {line width=0.1pt},
    semithick/.style=  {line width=0.5pt},
    thick/.style=      {line width=0.8pt},
    very thick/.style= {line width=1.2pt},
    ultra thick/.style={line width=1.6pt}
}

\newcommand\Dim[2][t]{\text{\rm Dim}_{#1}#2}

\def\Item{\item\abovedisplayskip=0pt\abovedisplayshortskip=5pt~\vspace*{-\baselineskip}}

\begin{document}

\title[Row removal for diagrammatic Cherednik algebras] {An analogue of row removal for \\ diagrammatic Cherednik algebras}

\author[C.~Bowman]{C.~Bowman}
\email{C.D.Bowman@kent.ac.uk}
\address{
University of Kent, Canterbury
CT2 7NF, UK}

\author{L.~Speyer}
\email{l.speyer@virginia.edu}
\address{University of Virginia,
Charlottesville, VA
22904,
USA}

\begin{abstract}
We prove an analogue of James--Donkin row removal theorems for   diagrammatic Cherednik algebras.
This is one of the first results concerning the (graded) decomposition numbers of these algebras over fields of arbitrary characteristic.
As a special case, our results yield a new reduction theorem for graded decomposition numbers and extension groups for cyclotomic $q$-Schur algebras.
\end{abstract}

\maketitle

\!\!\!\!\!\!\!\!\!\!\!\!\!\!\!\!
\section{Introduction}

Cyclotomic quiver Hecke algebras and (diagrammatic) Cherednik algebras are of central interest in Khovanov homology, knot theory, group theory, and higher representation theory.
Over the complex numbers, these algebras has been extensively studied by the great and the good of (geometric) representation theory \cite{losev, RSVV, Webster}.

Over fields of positive characteristic, the waters become muddier.
The quiver Hecke algebras and diagrammatic Cherednik algebras continue to be intimately related through Schur--Weyl duality \cite{Webster} and to be of fundamental interest to representation theorists \cite{bkisom,bk09,m14surv,bcs15,bs15,el17} (whereas the Cherednik algebras themselves diverge from this picture and become less mainstream).
In the case of the symmetric groups, our understanding of the quiver Hecke and diagrammatic Cherednik algebras has recently ballooned thanks to geometric insights of Williamson~\cite{w17} and others.
In this paper and~\cite{bs15}, the authors initiate the uniform study of all diagrammatic Cherednik algebras over arbitrary fields.

While the (quiver) Hecke algebras should be familiar to many readers (having enjoyed much direct study since the 1980s), the diagrammatic Cherednik algebras might appear  more mysterious and newfangled.
Associated to each cyclotomic quiver Hecke algebra $H_n(\kappa)$, we have a family of diagrammatic Cherednik algebras $A(n,\theta,\kappa)$ for $\theta\in \ZZ^\ell$, each of which gives us a different $\theta$-lens through which to view $H_n(\kappa)$.
These various $\theta$-lenses correspond to different parameterisations of simple modules under Ariki's categorification theorem (over $\CC$) and to varying the Lusztig $a$-function on $H_n(\kappa)$.
These $\theta$-lenses provide many different graded filtrations of projective $H_n(\kappa)$-modules and each $H_n(\kappa)$ admits many different  graded decomposition matrices (which all specialise to the same ordinary decomposition matrix by setting $t\to1$).
In~\cite{Webster}, Webster showed that each $A(n,\theta,\kappa)$ categorifies a $\theta$-twisted higher level Fock space (thus vastly generalising~\cite{kk12,WebsterKnots}).
It is the desire to fully understand these different $\theta$-structures on $H_n(\kappa)$ which has continued to inspire the work of Bonnaf\'e, Chlouveraki, Geck, Gordon, Griffeth, Jacon, and Rouquier, amongst many others over the past twenty years~\cite{BONNAFEROUQ,BONNAFEROUQ2,cgg12,cj12,Geck98,gj11,gr01,Jacon05}.
The diagrammatic Cherednik algebras provide us with a vast new array of tools with which to study the modular representation theory of $H_n(\kappa)$ through its various $\theta$-structures.

To each weighting $\theta \in \ZZ^\ell$  Webster introduces a graphical calculus for the set of   $\ell$-multipartitions based on an embedding of these multipartitions into $\RR^2$.
We define {\sf diagonal cuts} on these multipartitions in a graphical fashion; these cuts allow us to identify a pair of multipartitions $(\la, \mu)$ with two distinct pairs of multipartitions $(\la^L,\mu^L)$ and $(\la^R,\mu^R)$ (the left and right pieces of the cut) of smaller degree.
We thus reduce the problem of calculating the graded decomposition numbers of these algebras as follows.

\begin{Main} 
Let $R$ be an arbitrary field.  
Let $(\la,\mu)$ be a pair of $\ell$-multipartitions of $n$ and let $a \in \RR$.
If $(\la,\mu)$ admits a $\theta$-diagonal cut at $x=a$ into two pieces $(\la^L,\mu^L)$ and $(\la^R,\mu^R)$, then we can factorise the graded decomposition numbers for these algebras as
\[
d^{A(n,\theta,\kappa)}_{\la \mu}(t)= d^{A(n_L,\theta,\kappa)}_{\la^L \mu^L}(t) \times d^{A(n_R,\theta,\kappa)}_{\la^R \mu^R}(t)
\]
and the (graded) higher extension groups as 
\[
\Ext^k_{A(n,\theta,\kappa)} (\Delta(\la), \Delta(\mu)) \cong 
\bigoplus_{i+j=k}
\Ext^i_{A(n_L,\theta,\kappa)} (\Delta(\la^L), \Delta(\mu^L)) \otimes 
\Ext^j_{A(n_R,\theta,\kappa)} (\Delta(\la^R), \Delta(\mu^R)),
\]
\[
\Ext^k_{A(n,\theta,\kappa)} (\Delta(\la), L(\mu)) \cong 
\bigoplus_{i+j=k}
\Ext^i_{A(n_L,\theta,\kappa)} (\Delta(\la^L), L(\mu^L)) \otimes 
\Ext^j_{A(n_R,\theta,\kappa)} (\Delta(\la^R), L(\mu^R)),
\]
where $n_L=|\la^L|=|\mu^L|$ and $n_R=|\la^R|=|\mu^R|$.
\end{Main}

In the level $\ell=1$ case, the algebra $A(n,\theta,\kappa)$ is Morita equivalent to the classical $q$-Schur algebra.
In this case, one can make a diagonal cut across a pair of partitions if and only if one can make a {\sf vertical} (generalised column) cut, if and only if one can make a {\sf horizontal} (generalised row) cut.
These vertical and horizontal cuts across pairs of partitions have already been extensively studied; the analogue of the above result for (graded) decomposition numbers was treated in~\cite{j81, donkinnote,cmt02} and for $\Ext$ groups in~\cite{fl03,lm05,donkinhandbook}.

In higher levels, the (graphically defined) dominance orders on multipartitions are more exotic; our diagonal cuts may pass through many components of the multipartitions at once, as illustrated shortly.
Our diagonal cuts provide new information even in the case of the cyclotomic \mbox{$q$-Schur} algebras of Dipper, James and Mathas \cite{djm} (which are Morita equivalent to the diagrammatic Cherednik algebras with {\sf well-separated} weightings); see \cref{cylotomic}.

Roughly speaking, a pair of multipartitions $\la$, $\mu$ admits a diagonal cut at $x=a$ if when we draw the line $x = a$ on the Young diagrams for $\la$ and $\mu$ , we have the same number of boxes to the left of the line in $\la$ as in $\mu$, and likewise to the right of the line.
This concept is introduced more concretely in \cref{cutdef}.
In order to clarify the above, let's consider an example.
For $\theta= (0,1)$, the bipartitions
\[
\la = ((11,9,7,3^2,2,1^3),(9,4,2,1^4)) \quad \text{and} \quad \mu = ((10,9,8,4,3,1^5),(8,4,2,1^4))
\]
admit a {\sf $\theta$-diagonal cut} at $x = 5.2$ (note that we draw boxes with diagonals of length $2\ell=4$).  
To see this, we draw the bipartitions with respect to this weighting as in \cref{level1}.

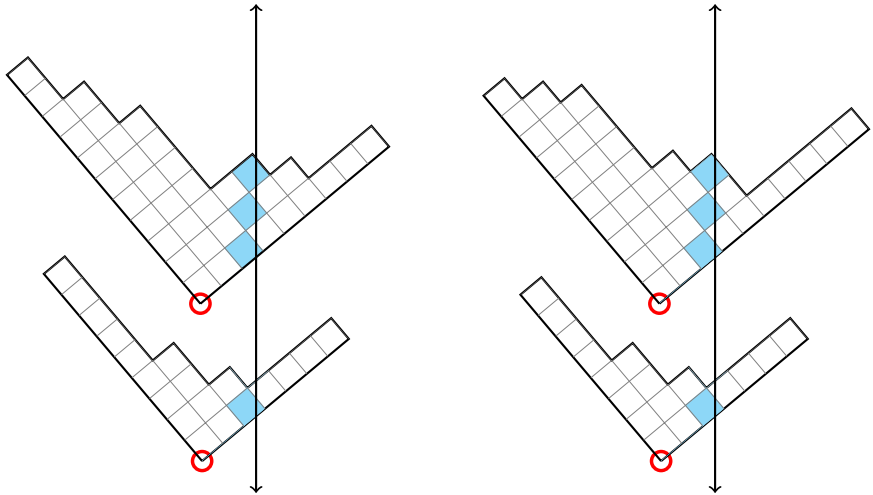
\begin{figure}[h]
\begin{center}\scalefont{0.6}
\begin{tikzpicture}[scale=1.8]
\begin{scope}
{\path (0,0) coordinate (origin);
\draw[wei2] (0,0)   circle (2pt);
  \draw[fill,cyan!40] (origin) 
   --++(40:2*0.2)
   --++(130:1*0.2)
    --++(40:1*0.2)
   --++(130:1*0.2)
    --++(40:1*0.2)
 --++(130:1*0.2)
    --++(40:1*0.2)
--++(-50:1*0.2)
    --++(220:1*0.2)--++(-50:1*0.2)
    --++(220:1*0.2)--++(-50:1*0.2)
    --++(220:1*0.2);

  \draw[thick] (origin) 
   --++(130:11*0.2)
   --++(40:1*0.2)
  --++(-50:2*0.2)	
  --++(40:1*0.2)
  --++(-50:2*0.2)
  --++(40:1*0.2)
  --++(-50:4*0.2)
  --++(40:2*0.2)
  --++(-50:1*0.2)
  --++(40:1*0.2)
  --++(-50:1*0.2)
  --++(40:3*0.2)
  --++(-50:1*0.2)
    --++(220:9*0.2);
     \clip (origin) 
     --++(130:11*0.2)
   --++(40:1*0.2)
  --++(-50:2*0.2)	
  --++(40:1*0.2)
  --++(-50:2*0.2)
  --++(40:1*0.2)
  --++(-50:4*0.2)
  --++(40:2*0.2)
  --++(-50:1*0.2)
  --++(40:1*0.2)
  --++(-50:1*0.2)
  --++(40:3*0.2)
  --++(-50:1*0.2)
    --++(220:9*0.2);
    \path (40:1cm) coordinate (A1);
  \path (40:2cm) coordinate (A2);
  \path (40:3cm) coordinate (A3);
  \path (40:4cm) coordinate (A4);
  \path (130:1cm) coordinate (B1);
  \path (130:2cm) coordinate (B2);
  \path (130:3cm) coordinate (B3);
  \path (130:4cm) coordinate (B4);
  \path (A1) ++(130:3cm) coordinate (C1);
  \path (A2) ++(130:2cm) coordinate (C2);
  \path (A3) ++(130:1cm) coordinate (C3);
   \foreach \i in {1,...,19}
  {
    \path (origin)++(40:0.2*\i cm)  coordinate (a\i);
    \path (origin)++(130:0.2*\i cm)  coordinate (b\i);
    \path (a\i)++(130:4cm) coordinate (ca\i);
    \path (b\i)++(40:4cm) coordinate (cb\i);
    \draw[thin,gray] (a\i) -- (ca\i)  (b\i) -- (cb\i); } 
  }
\end{scope}
\begin{scope}
{\path (0,-1.3)++(40:0.2*0.5)++(130:0.2*0.5) coordinate (origin);  
   \draw[wei2]  (0,-1.3)++(40:0.2*0.5)++(130:0.2*0.5)  circle (2pt);
   \draw[thick] (origin) 
   --++(130:9*0.2)
   --++(40:1*0.2)
  --++(-50:5*0.2)	
  --++(40:1*0.2)
  --++(-50:2*0.2)
  --++(40:1*0.2)
  --++(-50:1*0.2)
  --++(40:4*0.2)
  --++(-50:1*0.2)
     --++(220:7*0.2);
     \clip (origin) 
      --++(130:9*0.2)
   --++(40:1*0.2)
  --++(-50:5*0.2)	
  --++(40:1*0.2)
  --++(-50:2*0.2)
  --++(40:1*0.2)
  --++(-50:1*0.2)
  --++(40:4*0.2)
  --++(-50:1*0.2)
     --++(220:7*0.2);  \draw[fill,cyan!40] (origin) 
   --++(40:2*0.2)
   --++(130:1*0.2)
    --++(40:1*0.2)
   --++(130:1*0.2)
    --++(40:1*0.2)
 --++(130:1*0.2)
    --++(40:1*0.2)
--++(-50:1*0.2)
    --++(220:1*0.2)--++(-50:1*0.2)
    --++(220:1*0.2)--++(-50:1*0.2)
    --++(220:1*0.2);

    \path (40:1cm) coordinate (A1);
  \path (40:2cm) coordinate (A2);
  \path (40:3cm) coordinate (A3);
  \path (40:4cm) coordinate (A4);
  \path (130:1cm) coordinate (B1);
  \path (130:2cm) coordinate (B2);
  \path (130:3cm) coordinate (B3);
  \path (130:4cm) coordinate (B4);
  \path (A1) ++(130:3cm) coordinate (C1);
  \path (A2) ++(130:2cm) coordinate (C2);
  \path (A3) ++(130:1cm) coordinate (C3);
   \foreach \i in {1,...,19}
  {
    \path (origin)++(40:0.2*\i cm)  coordinate (a\i);
    \path (origin)++(130:0.2*\i cm)  coordinate (b\i);
    \path (a\i)++(130:4cm) coordinate (ca\i);
    \path (b\i)++(40:4cm) coordinate (cb\i);
    \draw[thin,gray] (a\i) -- (ca\i)  (b\i) -- (cb\i); } 
  }  \end{scope}

  \path (origin)++(40:0.2*1.4 cm)++(-50:0.2*1.4 cm)++(-90:0.2cm)
 coordinate (XXXX);
 \path (XXXX)++(90:3.6cm) coordinate (XXXXY);
 
   \draw[<->,thick,black] 
    (XXXX)
    --
    (XXXXY);
\end{tikzpicture}
\quad\quad\quad\quad 
 \begin{tikzpicture}[scale=1.8]
  \begin{scope}
{   \path (0,0) coordinate (origin);     \draw[wei2] (0,0)   circle (2pt);
  \draw[thick] (origin) 
    --++(130:10*0.2)
   --++(40:1*0.2)
  --++(-50:1*0.2)	
  --++(40:1*0.2)
  --++(-50:1*0.2)--++(40:1*0.2)
   --++(-50:4*0.2)--++(40:1*0.2)  --++(-50:1*0.2)
  --++(40:1*0.2)
  --++(-50:2*0.2)
    --++(40:5*0.2)
  --++(-50:1*0.2)
    --++(220:10*0.2);

     \clip (origin) 
    --++(130:10*0.2)
   --++(40:1*0.2)
  --++(-50:1*0.2)	
  --++(40:1*0.2)
  --++(-50:1*0.2)--++(40:1*0.2)
   --++(-50:4*0.2)--++(40:1*0.2)  --++(-50:1*0.2)
  --++(40:1*0.2)
  --++(-50:2*0.2)
   --++(40:5*0.2)
  --++(-50:1*0.2)
    --++(220:10*0.2);
  \draw[fill,cyan!40] (origin) 
   --++(40:2*0.2)
   --++(130:1*0.2)
    --++(40:1*0.2)
   --++(130:1*0.2)
    --++(40:1*0.2)
 --++(130:1*0.2)
    --++(40:1*0.2)
--++(-50:1*0.2)
    --++(220:1*0.2)--++(-50:1*0.2)
    --++(220:1*0.2)--++(-50:1*0.2)
    --++(220:1*0.2);

   \path (40:1cm) coordinate (A1);
  \path (40:2cm) coordinate (A2);
  \path (40:3cm) coordinate (A3);
  \path (40:4cm) coordinate (A4);
  \path (130:1cm) coordinate (B1);
  \path (130:2cm) coordinate (B2);
  \path (130:3cm) coordinate (B3);
  \path (130:4cm) coordinate (B4);
  \path (A1) ++(130:3cm) coordinate (C1);
  \path (A2) ++(130:2cm) coordinate (C2);
  \path (A3) ++(130:1cm) coordinate (C3);
  \foreach \i in {1,...,19}
  {
    \path (origin)++(40:0.2*\i cm)  coordinate (a\i);
    \path (origin)++(130:0.2*\i cm)  coordinate (b\i);
    \path (a\i)++(130:4cm) coordinate (ca\i);
    \path (b\i)++(40:4cm) coordinate (cb\i);
    \draw[thin,gray] (a\i) -- (ca\i)  (b\i) -- (cb\i); } 
  }
  \end{scope}
 \begin{scope}{\path (0,-1.3)++(40:0.2*0.5)++(130:0.2*0.5) coordinate (origin);  
   \draw[wei2]  (0,-1.3)++(40:0.2*0.5)++(130:0.2*0.5)  circle (2pt);  \draw[thick] (origin) 
     --++(130:8*0.2)
   --++(40:1*0.2)
  --++(-50:4*0.2)	
  --++(40:1*0.2)
  --++(-50:2*0.2)
    --++(40:1*0.2)
  --++(-50:1*0.2)
    --++(40:4*0.2)
  --++(-50:1*0.2)
     --++(220:7*0.2);
     \clip (origin) 
    --++(130:8*0.2)
   --++(40:1*0.2)
  --++(-50:4*0.2)	
  --++(40:1*0.2)
  --++(-50:2*0.2)
    --++(40:1*0.2)
  --++(-50:1*0.2)
    --++(40:4*0.2)
  --++(-50:1*0.2)
     --++(220:7*0.2);
       \draw[fill,cyan!40] (origin) 
   --++(40:2*0.2)
   --++(130:1*0.2)
    --++(40:1*0.2)
   --++(130:1*0.2)
    --++(40:1*0.2)
 --++(130:1*0.2)
    --++(40:1*0.2)
--++(-50:1*0.2)
    --++(220:1*0.2)--++(-50:1*0.2)
    --++(220:1*0.2)--++(-50:1*0.2)
    --++(220:1*0.2);

   \path (40:1cm) coordinate (A1);
  \path (40:2cm) coordinate (A2);
  \path (40:3cm) coordinate (A3);
  \path (40:4cm) coordinate (A4);
  \path (130:1cm) coordinate (B1);
  \path (130:2cm) coordinate (B2);
  \path (130:3cm) coordinate (B3);
  \path (130:4cm) coordinate (B4);
  \path (A1) ++(130:3cm) coordinate (C1);
  \path (A2) ++(130:2cm) coordinate (C2);
  \path (A3) ++(130:1cm) coordinate (C3);
  \foreach \i in {1,...,19}
  {
    \path (origin)++(40:0.2*\i cm)  coordinate (a\i);
    \path (origin)++(130:0.2*\i cm)  coordinate (b\i);
    \path (a\i)++(130:4cm) coordinate (ca\i);
    \path (b\i)++(40:4cm) coordinate (cb\i);
    \draw[thin,gray] (a\i) -- (ca\i)  (b\i) -- (cb\i); } 
  }  \end{scope}

  \path (origin)++(40:0.2*1.4 cm)++(-50:0.2*1.4 cm)++(-90:0.2cm)
 coordinate (XXXX);
 \path (XXXX)++(90:3.6cm) coordinate (XXXXY);
 
   \draw[<->,thick,black] 
    (XXXX)
    --
    (XXXXY);
\end{tikzpicture}
\end{center}
\caption{ 
The bipartition $\la = ((11,9,7,3^2,2,1^3), (9,4,2,1^4))$ on the left and the bipartition $\mu = ((10,9,8,4,3,1^5),(8,4,2,1^4))$ on the right, with $\theta = (0,1)$.  
The black line denotes the vertical cut though the line $x=5.2$. 
  We have also shaded in the longest diagonal of boxes which the cut goes through.  
}
\label{level1}
\end{figure}

Intuitively, the left- and right-hand sides of these bipartitions with respect to this cut are the `smallest' bipartitions which contain both the boxes intersected by the line and all the boxes to the left (respectively right). For $\la$ and $\mu$ as above, these cuts are depicted in \cref{level2}.
In this case,
\begin{alignat*}2
\la^L &= ((11,9,7,3^2), (9,4,2)),
&\qquad
\la^R &= ((3^5,2,1^3),(1^7)),\\
\mu^L &= ((10,9,8,4,3),(8,4,2)),
&\qquad
\mu^R &= ((3^5,1^5),(1^7)).
\end{alignat*}
Now, applying our main theorem to this example, we obtain
\[
d_{\la \mu}(t)= d_{\la^L \mu^L}(t) \times d_{\la^R \mu^R}(t).
\]
Now let $e=5$ and $\kappa=(0,2)$.
By results from \cite{bs15} (and the above) we obtain
\[
d_{\la\mu}(t) = (t^5 + t^3) \times t^2 = t^7 + t^5.
\]

\begin{figure}[h]
\begin{center}\scalefont{0.6}
\begin{tikzpicture}[scale=1.8]
\clip(-2,-1.5) rectangle (1.5,2.25); \begin{scope}
{\path (0,0) coordinate (origin);
\draw[wei2] (0,0)   circle (2pt);
  \draw[thick] (origin)
   --++(130:11*0.2)
   --++(40:1*0.2)
  --++(-50:2*0.2)	
  --++(40:1*0.2)
  --++(-50:2*0.2)
  --++(40:1*0.2)
  --++(-50:4*0.2)
  --++(40:2*0.2)
  --++(-50:3*0.2)
     --++(220:5*0.2);
     \clip (origin) 
       --++(130:11*0.2)
   --++(40:1*0.2)
  --++(-50:2*0.2)	
  --++(40:1*0.2)
  --++(-50:2*0.2)
  --++(40:1*0.2)
  --++(-50:4*0.2)
  --++(40:2*0.2)
  --++(-50:3*0.2)
     --++(220:5*0.2);  \draw[fill,cyan!40] (origin) 
   --++(40:2*0.2)
   --++(130:1*0.2)
    --++(40:1*0.2)
   --++(130:1*0.2)
    --++(40:1*0.2)
 --++(130:1*0.2)
    --++(40:1*0.2)
--++(-50:1*0.2)
    --++(220:1*0.2)--++(-50:1*0.2)
    --++(220:1*0.2)--++(-50:1*0.2)
    --++(220:1*0.2);

   \path (40:1cm) coordinate (A1);
  \path (40:2cm) coordinate (A2);
  \path (40:3cm) coordinate (A3);
  \path (40:4cm) coordinate (A4);
  \path (130:1cm) coordinate (B1);
  \path (130:2cm) coordinate (B2);
  \path (130:3cm) coordinate (B3);
  \path (130:4cm) coordinate (B4);
  \path (A1) ++(130:3cm) coordinate (C1);
  \path (A2) ++(130:2cm) coordinate (C2);
  \path (A3) ++(130:1cm) coordinate (C3);
  \foreach \i in {1,...,19}
  {
    \path (origin)++(40:0.2*\i cm)  coordinate (a\i);
    \path (origin)++(130:0.2*\i cm)  coordinate (b\i);
    \path (a\i)++(130:4cm) coordinate (ca\i);
    \path (b\i)++(40:4cm) coordinate (cb\i);
    \draw[thin,gray] (a\i) -- (ca\i)  (b\i) -- (cb\i); } 
  }
  \end{scope}
  \begin{scope}
{   \path (0,-1.3)++(40:0.2*0.5)++(130:0.2*0.5) coordinate (origin);  
   \draw[wei2]  (0,-1.3)++(40:0.2*0.5)++(130:0.2*0.5)  circle (2pt);
  \draw[thick] (origin) 
   --++(130:9*0.2)
   --++(40:1*0.2)
  --++(-50:5*0.2)	
  --++(40:1*0.2)
  --++(-50:2*0.2)
  --++(40:1*0.2)
  --++(-50:1*0.2)
   --++(-50:1*0.2)
     --++(220:3*0.2); 
     \clip (origin) 
    --++(130:9*0.2)
   --++(40:1*0.2)
  --++(-50:5*0.2)	
  --++(40:1*0.2)
  --++(-50:2*0.2)
  --++(40:1*0.2)
  --++(-50:1*0.2)
   --++(-50:1*0.2)
     --++(220:3*0.2);   \draw[fill,cyan!40] (origin) 
   --++(40:2*0.2)
   --++(130:1*0.2)
    --++(40:1*0.2)
   --++(130:1*0.2)
    --++(40:1*0.2)
 --++(130:1*0.2)
    --++(40:1*0.2)
--++(-50:1*0.2)
    --++(220:1*0.2)--++(-50:1*0.2)
    --++(220:1*0.2)--++(-50:1*0.2)
    --++(220:1*0.2);

   \path (40:1cm) coordinate (A1);
  \path (40:2cm) coordinate (A2);
  \path (40:3cm) coordinate (A3);
  \path (40:4cm) coordinate (A4);
  \path (130:1cm) coordinate (B1);
  \path (130:2cm) coordinate (B2);
  \path (130:3cm) coordinate (B3);
  \path (130:4cm) coordinate (B4);
  \path (A1) ++(130:3cm) coordinate (C1);
  \path (A2) ++(130:2cm) coordinate (C2);
  \path (A3) ++(130:1cm) coordinate (C3);
  \foreach \i in {1,...,19}
  {
    \path (origin)++(40:0.2*\i cm)  coordinate (a\i);
    \path (origin)++(130:0.2*\i cm)  coordinate (b\i);
    \path (a\i)++(130:4cm) coordinate (ca\i);
    \path (b\i)++(40:4cm) coordinate (cb\i);
    \draw[thin,gray] (a\i) -- (ca\i)  (b\i) -- (cb\i); } 
  }  \end{scope}

 \path (origin)++(40:0.2*1.4 cm)++(-50:0.2*1.4 cm)++(-90:0.3cm)
 coordinate (XXXX);

 \path (XXXX)++(90:3.6cm) coordinate (XXXXY);
 
   \draw[<->,thick,black] 
    (XXXX)
    --
    (XXXXY);
\end{tikzpicture}
\quad
\begin{tikzpicture}[scale=1.8]
\clip(-1.3,-1.5) rectangle (1.5,2.25); 

  \begin{scope}
{   \path (0,0) coordinate (origin);     \draw[wei2] (0,0)   circle (2pt);
  \draw[thick] (origin) 
    --++(130:10*0.2)
   --++(40:1*0.2)
  --++(-50:1*0.2)	
  --++(40:1*0.2)
  --++(-50:1*0.2)--++(40:1*0.2)
   --++(-50:4*0.2)--++(40:1*0.2)  --++(-50:1*0.2)
   --++(40:1*0.2)
  --++(-50:3*0.2)
    --++(220:5*0.2);

     \clip (origin) 
       --++(130:10*0.2)
   --++(40:1*0.2)
  --++(-50:1*0.2)	
  --++(40:1*0.2)
  --++(-50:1*0.2)--++(40:1*0.2)
   --++(-50:4*0.2)--++(40:1*0.2)  --++(-50:1*0.2)
   --++(40:1*0.2)
  --++(-50:3*0.2)
    --++(220:5*0.2);  \draw[fill,cyan!40] (origin) 
   --++(40:2*0.2)
   --++(130:1*0.2)
    --++(40:1*0.2)
   --++(130:1*0.2)
    --++(40:1*0.2)
 --++(130:1*0.2)
    --++(40:1*0.2)
--++(-50:1*0.2)
    --++(220:1*0.2)--++(-50:1*0.2)
    --++(220:1*0.2)--++(-50:1*0.2)
    --++(220:1*0.2);

   \path (40:1cm) coordinate (A1);
  \path (40:2cm) coordinate (A2);
  \path (40:3cm) coordinate (A3);
  \path (40:4cm) coordinate (A4);
  \path (130:1cm) coordinate (B1);
  \path (130:2cm) coordinate (B2);
  \path (130:3cm) coordinate (B3);
  \path (130:4cm) coordinate (B4);
  \path (A1) ++(130:3cm) coordinate (C1);
  \path (A2) ++(130:2cm) coordinate (C2);
  \path (A3) ++(130:1cm) coordinate (C3);
  \foreach \i in {1,...,19}
  {
    \path (origin)++(40:0.2*\i cm)  coordinate (a\i);
    \path (origin)++(130:0.2*\i cm)  coordinate (b\i);
    \path (a\i)++(130:4cm) coordinate (ca\i);
    \path (b\i)++(40:4cm) coordinate (cb\i);
    \draw[thin,gray] (a\i) -- (ca\i)  (b\i) -- (cb\i); } 
  }
  \end{scope}
 \begin{scope}{\path (0,-1.3)++(40:0.2*0.5)++(130:0.2*0.5) coordinate (origin);  
   \draw[wei2]  (0,-1.3)++(40:0.2*0.5)++(130:0.2*0.5)  circle (2pt); 
    \draw[thick] (origin) 
     --++(130:8*0.2)
   --++(40:1*0.2)
  --++(-50:4*0.2)	
  --++(40:1*0.2)
  --++(-50:2*0.2)
  --++(40:1*0.2)
   --++(-50:1*0.2)
   --++(-50:1*0.2)
     --++(220:3*0.2); 
     \clip (origin) 
        --++(130:8*0.2)
   --++(40:1*0.2)
  --++(-50:4*0.2)	
  --++(40:1*0.2)
  --++(-50:2*0.2)
  --++(40:1*0.2)
    --++(-50:1*0.2)
   --++(-50:1*0.2)
     --++(220:3*0.2);   \draw[fill,cyan!40] (origin) 
   --++(40:2*0.2)
   --++(130:1*0.2)
    --++(40:1*0.2)
   --++(130:1*0.2)
    --++(40:1*0.2)
 --++(130:1*0.2)
    --++(40:1*0.2)
--++(-50:1*0.2)
    --++(220:1*0.2)--++(-50:1*0.2)
    --++(220:1*0.2)--++(-50:1*0.2)
    --++(220:1*0.2);

   \path (40:1cm) coordinate (A1);
  \path (40:2cm) coordinate (A2);
  \path (40:3cm) coordinate (A3);
  \path (40:4cm) coordinate (A4);
  \path (130:1cm) coordinate (B1);
  \path (130:2cm) coordinate (B2);
  \path (130:3cm) coordinate (B3);
  \path (130:4cm) coordinate (B4);
  \path (A1) ++(130:3cm) coordinate (C1);
  \path (A2) ++(130:2cm) coordinate (C2);
  \path (A3) ++(130:1cm) coordinate (C3);
  \foreach \i in {1,...,19}
  {
    \path (origin)++(40:0.2*\i cm)  coordinate (a\i);
    \path (origin)++(130:0.2*\i cm)  coordinate (b\i);
    \path (a\i)++(130:4cm) coordinate (ca\i);
    \path (b\i)++(40:4cm) coordinate (cb\i);
    \draw[thin,gray] (a\i) -- (ca\i)  (b\i) -- (cb\i); } 
  }  \end{scope}

    \path (origin)++(40:0.2*1.4 cm)++(-50:0.2*1.4 cm)++(-90:0.2cm)
 coordinate (XXXX);
 \path (XXXX)++(90:3.6cm) coordinate (XXXXY);
 
   \draw[<->,thick,black] 
    (XXXX)
    --
    (XXXXY);
\end{tikzpicture}
\end{center}
\caption{
With bipartitions $\la$ and $\mu$ as in \cref{level1} and a diagonal cut at $x = 5.2$, we arrive at pairs of bipartitions $(\la^L, \mu^L)$  as depicted above (see below for   $(\la^R, \mu^R)$).}\label{level2}
\end{figure}
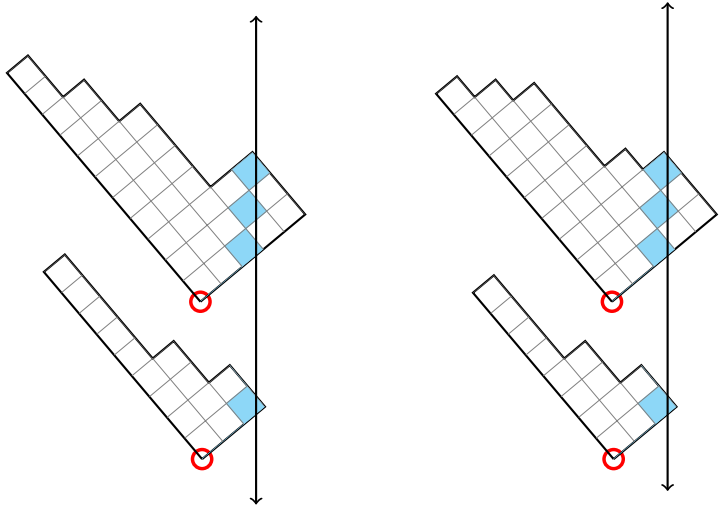

\begin{figure}[ht!]
$$
\scalefont{0.6}
\begin{tikzpicture}[scale=1.8]
\clip(-2,-1.5) rectangle (1.5,2.25);  \begin{scope}
{   \path (0,0) coordinate (origin);     \draw[wei2] (0,0)   circle (2pt);
  \draw[thick] (origin) 
   --++(130:3*0.2)
  --++(40:5*0.2)
  --++(-50:1*0.2)
  --++(40:1*0.2)
  --++(-50:1*0.2)
  --++(40:3*0.2)
  --++(-50:1*0.2)
    --++(220:9*0.2);
     \clip (origin) 
    --++(130:3*0.2)
  --++(40:5*0.2)
  --++(-50:1*0.2)
  --++(40:1*0.2)
  --++(-50:1*0.2)
  --++(40:3*0.2)
  --++(-50:1*0.2)
    --++(220:9*0.2);  \draw[fill,cyan!40] (origin) 
   --++(40:2*0.2)
   --++(130:1*0.2)
    --++(40:1*0.2)
   --++(130:1*0.2)
    --++(40:1*0.2)
 --++(130:1*0.2)
    --++(40:1*0.2)
--++(-50:1*0.2)
    --++(220:1*0.2)--++(-50:1*0.2)
    --++(220:1*0.2)--++(-50:1*0.2)
    --++(220:1*0.2);

   \path (40:1cm) coordinate (A1);
  \path (40:2cm) coordinate (A2);
  \path (40:3cm) coordinate (A3);
  \path (40:4cm) coordinate (A4);
  \path (130:1cm) coordinate (B1);
  \path (130:2cm) coordinate (B2);
  \path (130:3cm) coordinate (B3);
  \path (130:4cm) coordinate (B4);
  \path (A1) ++(130:3cm) coordinate (C1);
  \path (A2) ++(130:2cm) coordinate (C2);
  \path (A3) ++(130:1cm) coordinate (C3);
  \foreach \i in {1,...,19}
  {
    \path (origin)++(40:0.2*\i cm)  coordinate (a\i);
    \path (origin)++(130:0.2*\i cm)  coordinate (b\i);
    \path (a\i)++(130:4cm) coordinate (ca\i);
    \path (b\i)++(40:4cm) coordinate (cb\i);
    \draw[thin,gray] (a\i) -- (ca\i)  (b\i) -- (cb\i); } 
  }
  \end{scope}
  \begin{scope}
{   \path (0,-1.3)++(40:0.2*0.5)++(130:0.2*0.5) coordinate (origin);  
   \draw[wei2]  (0,-1.3)++(40:0.2*0.5)++(130:0.2*0.5)  circle (2pt);
  \draw[thick] (origin) 
   --++(130:1*0.2)
  --++(40:3*0.2)
  --++(40:4*0.2)
  --++(-50:1*0.2)
     --++(220:7*0.2);
     \clip (origin) 
   --++(130:1*0.2)
  --++(40:3*0.2)
  --++(40:4*0.2)
  --++(-50:1*0.2)
     --++(220:7*0.2);    \draw[fill,cyan!40] (origin) 
   --++(40:2*0.2)
   --++(130:1*0.2)
    --++(40:1*0.2)
   --++(130:1*0.2)
    --++(40:1*0.2)
 --++(130:1*0.2)
    --++(40:1*0.2)
--++(-50:1*0.2)
    --++(220:1*0.2)--++(-50:1*0.2)
    --++(220:1*0.2)--++(-50:1*0.2)
    --++(220:1*0.2);
   \path (40:1cm) coordinate (A1);
  \path (40:2cm) coordinate (A2);
  \path (40:3cm) coordinate (A3);
  \path (40:4cm) coordinate (A4);
  \path (130:1cm) coordinate (B1);
  \path (130:2cm) coordinate (B2);
  \path (130:3cm) coordinate (B3);
  \path (130:4cm) coordinate (B4);
  \path (A1) ++(130:3cm) coordinate (C1);
  \path (A2) ++(130:2cm) coordinate (C2);
  \path (A3) ++(130:1cm) coordinate (C3);
  \foreach \i in {1,...,19}
  {
    \path (origin)++(40:0.2*\i cm)  coordinate (a\i);
    \path (origin)++(130:0.2*\i cm)  coordinate (b\i);
    \path (a\i)++(130:4cm) coordinate (ca\i);
    \path (b\i)++(40:4cm) coordinate (cb\i);
    \draw[thin,gray] (a\i) -- (ca\i)  (b\i) -- (cb\i); } 
  }  \end{scope}

 \path (origin)++(40:0.2*1.4 cm)++(-50:0.2*1.4 cm)++(-90:0.3cm)
 coordinate (XXXX);

 \path (XXXX)++(90:3.6cm) coordinate (XXXXY);
 
   \draw[<->,thick,black] 
    (XXXX)
    --
    (XXXXY);
\end{tikzpicture}
\quad
\begin{tikzpicture}[scale=1.8]
  \clip(-1.3,-1.5) rectangle (1.5,2.25); \begin{scope}
{   \path (0,0) coordinate (origin);     \draw[wei2] (0,0)   circle (2pt);
  \draw[thick] (origin) 
    --++(130:3*0.2)
  --++(40:5*0.2)
  --++(-50:2*0.2)
   --++(40:5*0.2)
  --++(-50:1*0.2)
    --++(220:10*0.2);

     \clip (origin) 
  --++(130:3*0.2)
  --++(40:5*0.2)
  --++(-50:2*0.2)
   --++(40:5*0.2)
  --++(-50:1*0.2)
    --++(220:10*0.2);
  \draw[fill,cyan!40] (origin) 
   --++(40:2*0.2)
   --++(130:1*0.2)
    --++(40:1*0.2)
   --++(130:1*0.2)
    --++(40:1*0.2)
 --++(130:1*0.2)
    --++(40:1*0.2)
--++(-50:1*0.2)
    --++(220:1*0.2)--++(-50:1*0.2)
    --++(220:1*0.2)--++(-50:1*0.2)
    --++(220:1*0.2);

   \path (40:1cm) coordinate (A1);
  \path (40:2cm) coordinate (A2);
  \path (40:3cm) coordinate (A3);
  \path (40:4cm) coordinate (A4);
  \path (130:1cm) coordinate (B1);
  \path (130:2cm) coordinate (B2);
  \path (130:3cm) coordinate (B3);
  \path (130:4cm) coordinate (B4);
  \path (A1) ++(130:3cm) coordinate (C1);
  \path (A2) ++(130:2cm) coordinate (C2);
  \path (A3) ++(130:1cm) coordinate (C3);
  \foreach \i in {1,...,19}
  {
    \path (origin)++(40:0.2*\i cm)  coordinate (a\i);
    \path (origin)++(130:0.2*\i cm)  coordinate (b\i);
    \path (a\i)++(130:4cm) coordinate (ca\i);
    \path (b\i)++(40:4cm) coordinate (cb\i);
    \draw[thin,gray] (a\i) -- (ca\i)  (b\i) -- (cb\i); } 
  }
  \end{scope}
 \begin{scope}{\path (0,-1.3)++(40:0.2*0.5)++(130:0.2*0.5) coordinate (origin);  
   \draw[wei2]  (0,-1.3)++(40:0.2*0.5)++(130:0.2*0.5)  circle (2pt);  \draw[thick] (origin) 
     --++(130:1*0.2)
    --++(40:3*0.2)
    --++(40:4*0.2)
  --++(-50:1*0.2)
     --++(220:7*0.2);
     \clip (origin) 
        --++(130:1*0.2)
    --++(40:3*0.2)
    --++(40:4*0.2)
  --++(-50:1*0.2)
     --++(220:7*0.2);  \draw[fill,cyan!40] (origin) 
   --++(40:2*0.2)
   --++(130:1*0.2)
    --++(40:1*0.2)
   --++(130:1*0.2)
    --++(40:1*0.2)
 --++(130:1*0.2)
    --++(40:1*0.2)
--++(-50:1*0.2)
    --++(220:1*0.2)--++(-50:1*0.2)
    --++(220:1*0.2)--++(-50:1*0.2)
    --++(220:1*0.2);

   \path (40:1cm) coordinate (A1);
  \path (40:2cm) coordinate (A2);
  \path (40:3cm) coordinate (A3);
  \path (40:4cm) coordinate (A4);
  \path (130:1cm) coordinate (B1);
  \path (130:2cm) coordinate (B2);
  \path (130:3cm) coordinate (B3);
  \path (130:4cm) coordinate (B4);
  \path (A1) ++(130:3cm) coordinate (C1);
  \path (A2) ++(130:2cm) coordinate (C2);
  \path (A3) ++(130:1cm) coordinate (C3);
  \foreach \i in {1,...,19}
  {
    \path (origin)++(40:0.2*\i cm)  coordinate (a\i);
    \path (origin)++(130:0.2*\i cm)  coordinate (b\i);
    \path (a\i)++(130:4cm) coordinate (ca\i);
    \path (b\i)++(40:4cm) coordinate (cb\i);
    \draw[thin,gray] (a\i) -- (ca\i)  (b\i) -- (cb\i); } 
  }  \end{scope}

    \path (origin)++(40:0.2*1.4 cm)++(-50:0.2*1.4 cm)++(-90:0.2cm)
 coordinate (XXXX);
 \path (XXXX)++(90:3.6cm) coordinate (XXXXY);
 
   \draw[<->,thick,black] 
    (XXXX)
    --
    (XXXXY);
\end{tikzpicture}$$
\caption{
With bipartitions $\la$ and $\mu$ as in \cref{level1} and a diagonal cut at $x = 5.2$, we arrive at pairs of bipartitions $(\la^R, \mu^R)$  as depicted above (see below for   $(\la^L, \mu^L)$).}\label{level2a}
\end{figure}
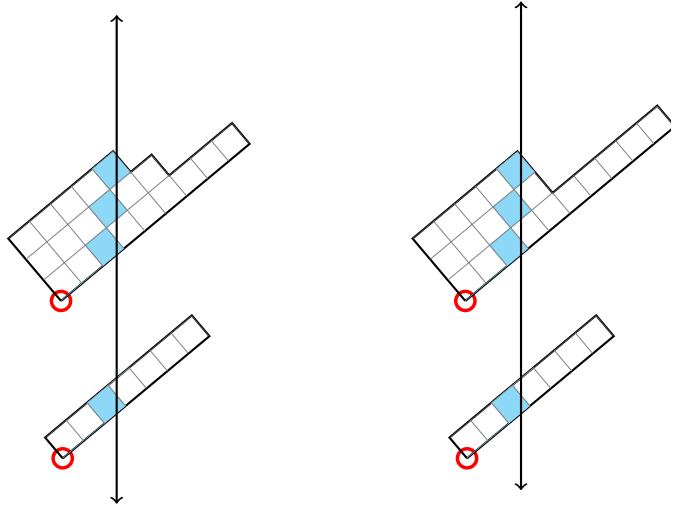

Finally, we remark that our main interest in the diagrammatic Cherednik algebras comes from the manner in which they control the representation theory of cyclotomic quiver Hecke algebras $H_n(\kappa)$ of affine type $A$ (over fields of arbitrary characteristic!).  
For each weighting, $\theta \in \ZZ^\ell$, we have a corresponding cellular structure on $H_n(\kappa)$, see \cite{bowman17} for details.
Given $\theta \in \ZZ^\ell$ a weighting, we let $\{\Delta^\theta(\la) \mid \la \in \mptn \ell n\}$ denote the set of cell modules and let $\{D^\theta(\mu) \mid \mu \in \Theta\subseteq \mptn \ell n\}$ denote the corresponding set of simple modules for $H_n(\kappa)$.
The graded decomposition matrix with respect to this cellular structure appears as a submatrix of the decomposition matrix of $A(n,\theta,\kappa)$ (see \cite[Corollary 4.3]{bowman17}) and we hence obtain the following corollary.

\begin{Main2}
Let $R$ be an arbitrary field.
Let $ \la\in \mptn \ell n$ , $\mu \in \Theta\subseteq \mptn \ell n$ and let $a \in \RR$.
If $(\la,\mu)$ admits a $\theta$-diagonal cut at $x=a$ into two pieces $(\la^L,\mu^L)$ and $(\la^R,\mu^R)$, then we can factorise the graded decomposition numbers with respect to the corresponding cellular structure of $H_n(\kappa)$ as follows.
\[
[\Delta^\theta(\la): D^\theta(\mu)\langle k \rangle ]  = 
\sum_{i+j=k }[\Delta^\theta(\la^L): D^\theta(\mu^L)\langle i \rangle] \times [\Delta^\theta(\la^R): D^\theta(\mu^R)\langle j \rangle]
\]
\end{Main2}

\begin{ack}
The authors would like to thank the Royal Commission for the Exhibition of 1851 and the Japan Society for the Promotion of Science for their financial support.
\end{ack}

\section{The diagrammatic Cherednik algebra}\label{lCasec}

In this section we define the diagrammatic Cherednik algebras and recall the combinatorics underlying their representation theory.
We fix an arbitrary integral domain $R$.
We will first prove a general result for relating a graded cellular algebra (with a highest weight theory) to certain
subquotient algebras.

\begin{defn}\label{defn1}
Suppose that $A$ is a $\ZZ$-graded $R$-algebra which is of finite rank over $R$.
We say that $A$ is a {\sf graded cellular algebra with a highest weight theory} if the following conditions hold.

The algebra is equipped with a {\sf cell datum} $(\Lambda,\TSStd,C,\degr)$, where $(\Lambda,\trianglerighteq )$ is the {\sf weight poset}.
For each $\la,\mu\in\Lambda$ such that $\la \trianglerighteq \mu$, we have a finite set, denoted $\TSStd(\la,\mu)$, and we let $\TSStd(\la) = \cup_{\mu \domby \la} \TSStd(\la,\mu)$.
There exist maps
\[
C:{\coprod_{\la\in\Lambda}\TSStd(\la)\times \TSStd(\la)}\to A 
  \quad\text{and}\quad
 \degr: {\coprod_{\la\in\Lambda}\TSStd(\la)} \to \ZZ
\]
such that $C$ is injective. We denote $C(\SSTS,\SSTT) = c^\la_{\SSTS\SSTT}$ for $\SSTS,\SSTT\in\TSStd(\la)$.  
We require that $A$ satisfies properties (1)--(6), below.  
\begin{enumerate}
\item Each element $c^\la_{\SSTS,\SSTT}$ is homogeneous of degree 
\[
\degr(c^\la_{\SSTS,\SSTT}) = \degr(\SSTS)+\degr(\SSTT),
\]
for $\la\in\Lambda$ and $\SSTS,\SSTT\in \TSStd(\la)$.
\item The set $\{c^\la_{\SSTS,\SSTT}\mid\SSTS,\SSTT \in \TSStd(\la), \, \la \in \Lambda \}$ is an $R$-basis of $A$.
\item  If $\SSTS,\SSTT\in \TSStd(\la)$, for some $\la\in\Lambda$, and $a\in A$ then there exist scalars $r_{\SSTS,\SSTU}(a)$, which do not depend on $\SSTT$, such that
\[
ac^\la_{\SSTS,\SSTT} = \; \sum_{\mathclap{\SSTU \in \TSStd(\la)}} \; r_{\SSTS,\SSTU}(a)c^\la_{\SSTU,\SSTT}\pmod {A^{\vartriangleright  \la}},
\]
where $A^{\vartriangleright \la}$ is the $R$-submodule of $A$ spanned by
\[
\{c^\mu_{\SSTQ,\SSTR}\mid\mu \vartriangleright  \la\text{ and }\SSTQ,\SSTR\in \TSStd(\mu )\}.
\]
\item The $R$-linear map $*:A\to A$ determined by $(c^\la_{\SSTS,\SSTT})^*=c^\la_{\SSTT,\SSTS}$, for all $\la\in\Lambda$ and all $\SSTS,\SSTT\in\TSStd(\la)$, is an anti-isomorphism of $A$.
\item The identity $1_A$ of $A$ has a decomposition $1_A = \sum_{\la \in \Lambda} 1_\la$ into pairwise orthogonal idempotents $1_\la$.
\item For $\SSTS\in\TSStd(\la,\mu)$, $\SSTT \in\TSStd(\la,\nu)$, we have that $1_\mu c^\la_{\SSTS,\SSTT}1_\nu =  c^\la_{\SSTS,\SSTT}$.
There exists a unique element $\SSTT^\la \in \TSStd(\la,\la)$, and $c_{\SSTT^\la,\SSTT^\la}^\la = 1_\la$.
\end{enumerate}
\end{defn}

Let $A$ be a graded cellular algebra with a highest weight theory, $R$ be a field, and $t$ be an indeterminate over $\ZZ_{\geq 0}$.
The {\sf graded decomposition matrix} of $A$ is the matrix $\Dec=(d_{\la\mu}(t))$, where for $\la,\mu\in\Lambda$ we have
\[
d_{\la\mu}(t)=\sum_{k\in\ZZ} [\Delta(\la):L(\mu)\langle k\rangle]\,t^k,
\]
where $\Lambda$ indexes the {\sf standard modules} $\Delta(\la)$ and their simple heads $L(\la)$, and $\langle k\rangle$ denotes a grading shift by $k$.

\begin{rem}
In a finite-dimensional (graded) cellular algebra over a field, all simple (graded) modules arise as heads of cell modules $\Delta(\la)$ (and their graded shifts).
Properties (5) and (6) in \cref{defn1} ensure that {\sf every} cell module gives rise to a simple head, so that the algebra is also quasi-hereditary.
This is why we refer to the cell modules as standard modules.
\end{rem}

\begin{defn}
Let  $Q\subseteq\Lambda$.
We say that $Q$ is {\sf saturated} if for any $\alpha\in Q$ and $\beta \in \Lambda$ with $\beta \vartriangleleft \alpha$, we have that $\beta \in Q$. We say that $Q$ is {\sf cosaturated} if its complement in $\Lambda$ is saturated.
\end{defn}

\begin{defn}\label{subquot}
Let $E$ and $F$ denote subsets of $\mptn \ell n$ which are saturated and co-saturated, respectively.
We let
\[
e = \sum_{\mu \in E \cap F} 1_\mu \quad \text{and} \quad
f = \sum_{\mu \in F\setminus E} 1_\mu
\]
in $A$.
We let $A_{E\cap F} $ denote the subquotient of $A$ given by
\[
A_{E\cap F} = e(A/(A f A))e.
\]
\end{defn}

\begin{prop}\label{3.6rmk}
The algebra $A_{E \cap F}$ is a graded cellular algebra with a highest weight theory.
The cellular basis is given by
\[
\{c^\la_{\SSTS,\SSTT} \mid \SSTS \in \mathcal{T}(\la,\mu), \, \SSTT\in\mathcal{T}(\la,\nu), \, \la, \mu, \nu \in E \cap F\},
\]
with respect to the partial order on $E\cap F \subseteq \Lambda$.
Moreover, if $R$ is a field, then
\[
d_{\la \mu}^{A }(t)= d_{\la \mu}^{A_{E\cap F}}(t)
\]
and, furthermore,
\[
\Ext^k_{A} (\Delta^A(\la), \Delta^A(\mu)) \cong
\Ext^k_{A_{E\cap F}} (\Delta^{A_{E\cap F}}(\la), \Delta^{A_{E\cap F}}(\mu)) \]\[
\Ext^k_{A} (\Delta^A(\la), L^A(\mu)) \cong
\Ext^k_{A_{E\cap F}} (\Delta^{A_{E\cap F}}(\la), L^{A_{E\cap F}}(\mu)) \] 
for all $k\geq 0$.
\end{prop}

\begin{proof}
The set $F\setminus E$ is cosaturated in the $\theta$-dominance ordering and so 
\[
\langle c^\la_{\SSTS,\SSTT} \mid \SSTS \in \mathcal{T}(\la, \mu), \, \SSTT\in \mathcal{T}(\la, \nu), \, \la \in F\setminus E, \mu,\nu\in\Lambda \rangle_{R}
\]
is an ideal of $A$.
The resulting quotient algebra has basis indexed by $\SSTS \in \mathcal{T}(\la,\mu)$, $ \SSTT\in \mathcal{T}(\la,\nu)$ such that $\la, \mu, \nu \not \in (F\setminus E)$. 
Applying the idempotent truncation to this basis 
we obtain the required basis of $A_{E \cap F}$ by condition $(6)$ of \cref{defn1}.
The graded decomposition numbers (as well as dimensions of higher extension groups) are preserved under both the quotient and truncation maps; this follows by the arguments of \cite[Appendix: Lemmas A3.1, A3.3 and A3.13]{Donkin} (for the ungraded case) as the ideal is generated by a (degree zero) idempotent.
\end{proof}

\subsection{Combinatorics of loadings}
We now recall the combinatorics underlying  complex reflection groups and their diagrammatic Cherednik algebras.
We have tried to keep this section brief, and we refer to \cite{bcs15,bowman17,bs15} for further details and many illustrative examples.

Fix integers $\ell ,n\in\ZZ_{\geq 0} $ and $e\in\{2,3,\dots\}\cup\{\infty\}$.
We define a {\sf weighting} $\theta = (\theta_1,\dots, \theta_\ell ) \in \ZZ^\ell$ to be any strictly increasing $\ell $-tuple such that $\theta_j-\theta_i \notin \ell \ZZ$ for $1 \leq i<j \leq \ell$.
Let $\kappa$ denote an $e$-{\sf multicharge} $\kappa = (\kappa_1,\dots,\kappa_\ell )\in (\ZZ/e\ZZ)^\ell$.

We define the corresponding {\sf Russian array} as follows.
For each $1\leq m\leq \ell$, we place a point on the real line at $\theta_m$ and consider the region bounded by  half-lines at angles $3\pi/4 $ and $\pi/4$.
We tile the resulting quadrant with a lattice of squares, each with diagonal of length $2 \g $.

An {\sf $\ell $-multipartition} $\la=(\la^{(1)},\dots,\la^{(\ell)})$ of $n$ is an $\ell$-tuple of partitions such that $|\la^{(1)}|+\dots+ |\la^{(\ell)}|=n$. We will denote the set of $\ell $-multipartitions of $n$ by $\mptn {\ell}n$.
Let $\la=(\la^{(1)},\la^{(2)},\ldots ,\la^{(\ell)}) \in \mptn {\ell}n$.
The {\sf $\theta$-Young diagram} $[\la]_\theta$ is defined to be the set
\[
[\la]_\theta = \{(r,c,m)\in\mathbb{N}\times\mathbb{N}\times\{1,\dots,\ell \} \mid c\leq \la^{(m)}_r\}.
\]
We refer to elements of the $\theta$-Young diagram as {\sf nodes} or {\sf boxes}.
We define the {\sf residue} of a node $(r, c, m) \in [\la]_\theta$ to be $\kappa_m+c-r \pmod e$.

\begin{defn} \label{domdef}
Let $(r,c,m), (r',c',m')$ be two $i$-boxes and $\theta\in \ZZ^\ell$ be a weighting.
We write $(r,c,m) \rhd_\theta  (r',c',m')$ if either
\begin{itemize}
\item[$(i)$] $\theta_m +\ell(r-c) <\theta_{m'} +\ell(r'-c')$ or

\item[$(ii)$] $\theta_m +\ell(r-c) = \theta_{m'} +\ell(r'-c')$
and $r+c < r'+c'$.
\end{itemize}
Given two multipartitions, $\la, \mu \in \mptn \ell n$, we say that $\la$ $\theta$-{\sf dominates} $\mu$ (and write $\mu \vartriangleleft_\theta \la$) if for every $i$-box $(r,c,m) \in [\mu]_\theta$, there exist strictly more $i$-boxes $(r',c',m') \in [\la]_\theta$ which $\theta$-dominate $(r,c,m)$ than there do $i$-boxes $(r',c',m') \in [\mu]_\theta$ which $\theta$-dominate $(r,c,m)$.
We extend this order to subsets (of the same cardinality) of multipartitions in the obvious fashion.
\end{defn}


We now provide a diagrammatic way of encoding the $\theta$-dominance order on $[\la]_\theta$.
For each node of $[\la]_\theta$ we draw a box in the Russian array.
We place the first node, $(1,1,m)$, of component $m$ at $\theta_m$ on the real line, with column numbers increasing going northwest from this node (so that a single row goes northwest), and row numbers increasing going northeast (so that a single column goes northeast).
 The diagram is tilted ever-so-slightly (by an angle of $\varepsilon \ll \tfrac{1}{2n}$ radians) in the clockwise direction so that the top vertex of the box $(r,c,m)$ (that is, the box in the $r$th row and $c$th column of the $m$th component of $ [\la]_\theta $) has $x$-coordinate $\mathbf{I}_{(r,c,m)}^\theta = \theta_m + \ell(r-c) + (r+c)\varepsilon$ (up to first order approximation).
Having made this tilt, we note that the boxes of a $\theta$-Young diagram of a multipartition all have distinct $x$-coordinates (by our assumptions on the weighting and tilt).

Given $\la \in \mptn \ell n$ we have an associated {\sf  loading} $\mathbf{I}_\la^\theta \subseteq \mathbb{R}$ given by the ordered set of real numbers obtained by  projecting the top vertex of each box $(r,c,m)\in [\la]_\theta$ to its $x$-coordinate $\mathbf{I}_{(r,c,m)}^\theta \in\RR$.  
The associated {\sf residue sequence}, ${\rm res}(\la)$, of $\la$ is given by reading the residues of the boxes of $\la$ according to the ordering on ${\bf I}_\la^\theta$.

\begin{rem}
When $\ell = 1$, the $\theta$-dominance order agrees with the usual dominance order on partitions.
\end{rem}


\begin{defn}
 Let $\la,\mu \in \mptn {\ell}n$.
A {\sf $\la$-tableau of weight $\mu$} is a bijective map $\SSTT: [\la]_\theta \to \mathbf{I}_\mu^\theta$ which respects residues.
In other words, $\SSTT$ maps a given node $(r,c,m)$ of  $[\la]_\theta$ to a point $\mathbf{I}_{(r',c',m')}^\theta \in \RR$ for ${(r',c',m')} \in [\mu]_\theta$ such that $\kappa_m+c-r =  \kappa_{m'}+c'-r' \pmod e$.
\end{defn}

\begin{defn}\label{hjkljkhjklh}
A $\la$-tableau, $\SSTT$, of shape $\la$ and weight $\mu$ is said to be {\sf semistandard} if
\begin{itemize}
\item $\SSTT(1,1,m)>\theta_m$,
\item $\SSTT(r,c,m)> \SSTT(r-1,c,m)  +\g$,
\item $\SSTT(r,c,m)> \SSTT(r,c-1,m) -\g$.
\end{itemize}
We  denote the set of all  semistandard tableaux of shape $\la$ and weight $\mu$ by $\SStd(\la,\mu)$.
\end{defn}

\subsection{The diagrammatic Cherednik algebra}\label{relationspageofstuff}

Recall that we have fixed $\ell, n\in \ZZ_{>0}$, and $e\in\{2,3,\dots\}\cup\{\infty\}$.
Given any weighting $\theta\in\ZZ^\ell$ and any $e$-multicharge $\kappa\in (\ZZ/e\ZZ)^\ell$, we now recall the definition of the diagrammatic Cherednik algebra, $A(n,\theta,\kappa)$.

\begin{defn}
We define a $\theta$-{\sf diagram} {of type} $G(\ell,1,n)$ to  be a {\sf frame} $\mathbb{R}\times [0,1]$ with distinguished solid points on the northern and southern boundaries given by the loadings $\mathbf{I}_\mu^\theta$ and $\mathbf{I}_\la^\theta$ for some $\la,\mu \in \mptn {\ell}n$ and a collection of solid strands each of which starts at a northern point, $\mathbf{I}_{(r,c,m)}^\theta$ say, and ends at a southern point, $\mathbf{I}_{(r',c',m')}^\theta$ such that $\kappa_m+c-r \equiv i \equiv \kappa_{m'}+c'-r' \pmod e$ (we refer to this as a {\sf solid $i$-strand}).
We further require that each strand has a mapping diffeomorphically to $[0,1]$ via the projection to the $y$-axis.
Each strand is allowed to carry any number of dots. We draw:
$(i)$ a dashed line $\ell$ units to the left of each strand, which we call a {\sf ghost $i$-strand} or $i$-{\sf ghost}; $(ii)$ vertical red lines at $\theta_m \in \RR$ each of which carries a residue $\kappa_m$ for $1\leq m\leq \ell$ which we call a {\sf red $\kappa_m$-strand}. 
 We require that there are no triple points or tangencies involving any combination of strands, ghosts or red lines and no dots lie on crossings.
We consider these diagrams equivalent if they are related by an isotopy that avoids these tangencies, double points and dots on crossings.
A $\theta$-diagram is depicted in \cref{diagram}, below. 
\end{defn}

\!\!\!\!
\begin{figure}[ht!]
\[
\begin{tikzpicture}[baseline, thick,yscale=1.1,xscale=4.5]\scalefont{0.8}
\draw[wei]  (-1.1, -2)  to[out=90,in=-90] (-1.1, 2);
\draw[red] (-1.1, -2.05) [below] node{$\ \kappa_2$};

  \draw[wei]  (-1.2, -2) to[out=90,in=-90] (-1.2, 2);
   \draw[red] (-1.2, -2.05) [below] node{$\ \kappa_1$};
 
 \draw(-1.5,-2) rectangle (1.4,2);

  \draw[gray, densely dotted] (-0.4+-.5,-2) to[out=90,in=-90]   (-0.4+-1,-1)  to[out=90,in=-90](-0.4+1,.3) to[out=90,in=-90]  
    (-0.4+0,2);
  \draw[gray, densely dotted] (-0.4+.5,-2) to[out=90,in=-90]  (-0.4+.5,0) to[out=90,in=-90] (-0.4+1,2);
  \draw[gray, densely dotted]  (-0.4+1,-2) to[out=90,in=-90]  (-0.4+-1,1) to[out=90,in=-90] (-0.4+.5,2);
  \draw[gray, densely dotted] (-0.4+0,-2) to[out=90,in=-90]  
  (-0.4+0.1,0) to[out=90,in=-90]
  (-0.4+-.5,2);
  \draw[gray, densely dotted]  (-0.4+-1, -2) to[out=95,in=-90]  
  (-0.4+-.5,-1.1) to[out=90,in=-90] (-0.4+-1,0) to[out=90,in=-90] (-0.4+-.5,1)
  to[out=90,in=-90]   (-0.4+-1,2);

  \draw (-.5,-2) to[out=90,in=-90] node[below,at start]{$\ i_2$} (-1,-1)  to[out=90,in=-90](1,.2) to[out=90,in=-90]    (0,2);
  \draw (.5,-2) to[out=90,in=-90] node[below,at start]{$\ i_4$} (.5,0) to[out=90,in=-90] node[midway,circle,fill=black,inner sep=2pt]{} (1,2);
  \draw  (1,-2) to[out=90,in=-90]  node[below,at start]{$\ i_5$} (-1,1) to[out=90,in=-90]
   (.5,2);
  \draw (0,-2) to[out=90,in=-90] node[below,at start]{$\ i_3$}node[midway,below,circle,fill=black,inner sep=2pt]{}
  (0.05,0) to[out=90,in=-90]
  (-.5,2);
  \draw  (-1, -2) to[out=90,in=-90] node[below,at start]{$\ i_1$}
  (-.5,-1) to[out=90,in=-90] (-1,0) to[out=90,in=-90] (-.5,1)
  to[out=90,in=-90]   (-1,2);
\end{tikzpicture} 
\]
\caption{A $\theta$-diagram  for $\theta=(0,1)$  with northern and southern loading
given by ${\bf I}_\omega$ where $\omega= (\varnothing,(1^5))$.}
\label{diagram}
\end{figure}
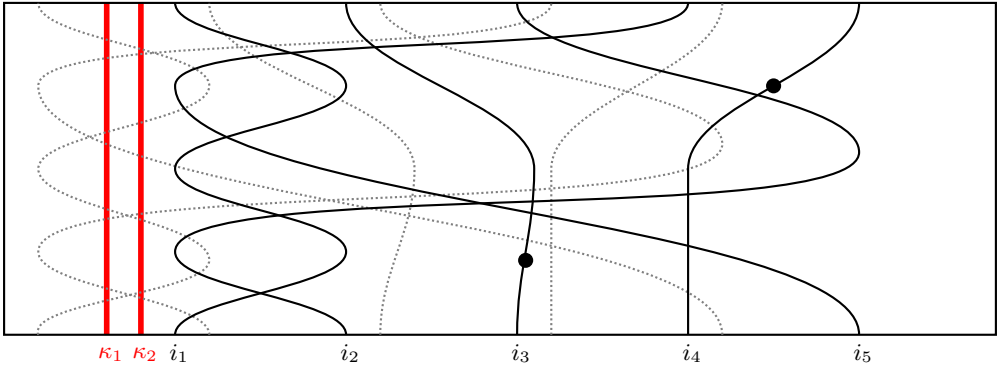

\begin{defn}[Definition 4.1 \cite{Webster}]
The {\sf diagrammatic Cherednik algebra}, $A(n,\theta,\kappa)$, is the $R$-algebra spanned by all $\theta$-diagrams modulo the following local relations (here a local relation means one that can be applied on a small region of the diagram)
\begin{enumerate}[label=(2.\arabic*)] 
\item\label{rel1}  Any diagram may be deformed isotopically; that is, by a continuous deformation of the diagram which at no point introduces or removes any crossings of strands.    \item\label{rel2} 
For $i\neq j$ we have that dots pass through crossings.
\[
\scalefont{0.8}\begin{tikzpicture}[scale=.5,baseline]
\draw[very thick](-4,0) +(-1,-1) -- +(1,1) node[below,at start]
  {$i$}; \draw[very thick](-4,0) +(1,-1) -- +(-1,1) node[below,at
  start] {$j$}; \fill (-4.5,.5) circle (5pt);
      \node at (-2,0){=}; \draw[very thick](0,0) +(-1,-1) -- +(1,1)
      node[below,at start] {$i$}; \draw[very thick](0,0) +(1,-1) --
      +(-1,1) node[below,at start] {$j$}; \fill (.5,-.5) circle (5pt);
      \node at (4,0){ };
\end{tikzpicture}
\]
\item\label{rel3} For two like-labelled strands, we get an error term.
\[
\scalefont{0.8}\begin{tikzpicture}[scale=.5,baseline]
      \draw[very thick](-4,0) +(-1,-1) -- +(1,1) node[below,at start]
      {$i$}; \draw[very thick](-4,0) +(1,-1) -- +(-1,1) node[below,at
      start] {$i$}; \fill (-4.5,.5) circle (5pt);
       \node at (-2,0){=}; \draw[very thick](0,0) +(-1,-1) -- +(1,1)
      node[below,at start] {$i$}; \draw[very thick](0,0) +(1,-1) --
      +(-1,1) node[below,at start] {$i$}; \fill (.5,-.5) circle (5pt);
      \node at (2,0){+}; \draw[very thick](4,0) +(-1,-1) -- +(-1,1)
      node[below,at start] {$i$}; \draw[very thick](4,0) +(0,-1) --
      +(0,1) node[below,at start] {$i$};
\end{tikzpicture}  \quad \quad \quad
\scalefont{0.8}\begin{tikzpicture}[scale=.5,baseline]
      \draw[very thick](-4,0) +(-1,-1) -- +(1,1) node[below,at start]
      {$i$}; \draw[very thick](-4,0) +(1,-1) -- +(-1,1) node[below,at
      start] {$i$}; \fill (-4.5,-.5) circle (5pt);
           \node at (-2,0){=}; \draw[very thick](0,0) +(-1,-1) -- +(1,1)
      node[below,at start] {$i$}; \draw[very thick](0,0) +(1,-1) --
      +(-1,1) node[below,at start] {$i$}; \fill (.5,.5) circle (5pt);
      \node at (2,0){+}; \draw[very thick](4,0) +(-1,-1) -- +(-1,1)
      node[below,at start] {$i$}; \draw[very thick](4,0) +(0,-1) --
      +(0,1) node[below,at start] {$i$};
\end{tikzpicture}
\]
\item\label{rel4} For double crossings of solid strands, we have the following.
\[
\scalefont{0.8}\begin{tikzpicture}[very thick,scale=0.5,baseline]
\draw (-2.8,-1) .. controls (-1.2,0) ..  +(0,2)
node[below,at start]{$i$};
\draw (-1.2,-1) .. controls (-2.8,0) ..  +(0,2) node[below,at start]{$i$};
\node at (-.5,0) {=};
\node at (0.4,0) {$0$};
\end{tikzpicture}
\hspace{.7cm}
\scalefont{0.8}\begin{tikzpicture}[very thick,scale=0.5,baseline]
\draw (-2.8,-1) .. controls (-1.2,0) ..  +(0,2)
node[below,at start]{$i$};
\draw (-1.2,-1) .. controls (-2.8,0) ..  +(0,2)
node[below,at start]{$j$};
\node at (-.5,0) {=}; 
\draw (1.8,-1) -- +(0,2) node[below,at start]{$j$};
\draw (1,-1) -- +(0,2) node[below,at start]{$i$}; 
\end{tikzpicture}
\]
\end{enumerate}
\begin{enumerate}[resume, label=(2.\arabic*)]  
\item\label{rel5} If $j\neq i-1$,  then we can pass ghosts through solid strands.
\[
\begin{tikzpicture}[very thick,xscale=1,yscale=0.5,baseline]
\draw (1,-1) to[in=-90,out=90]  node[below, at start]{$i$} (1.5,0) to[in=-90,out=90] (1,1);
\draw[densely dashed] (1.5,-1) to[in=-90,out=90] (1,0) to[in=-90,out=90] (1.5,1);
\draw (2.5,-1) to[in=-90,out=90]  node[below, at start]{$j$} (2,0) to[in=-90,out=90] (2.5,1);
\node at (3,0) {=};
\draw (3.7,-1) -- (3.7,1) node[below, at start]{$i$};
\draw[densely dashed] (4.2,-1) to (4.2,1);
\draw (5.2,-1) -- (5.2,1) node[below, at start]{$j$};
\end{tikzpicture} \quad\quad \quad \quad
\scalefont{0.8}\begin{tikzpicture}[very thick,xscale=1,yscale=0.5,baseline]
\draw[densely dashed] (1,-1) to[in=-90,out=90] (1.5,0) to[in=-90,out=90] (1,1);
\draw (1.5,-1) to[in=-90,out=90] node[below, at start]{$i$} (1,0) to[in=-90,out=90] (1.5,1);
\draw (2,-1) to[in=-90,out=90]  node[below, at start]{$\tiny j$} (2.5,0) to[in=-90,out=90] (2,1);
\node at (3,0) {=};
\draw (4.2,-1) -- (4.2,1) node[below, at start]{$i$};
\draw[densely dashed] (3.7,-1) to (3.7,1);
\draw (5.2,-1) -- (5.2,1) node[below, at start]{$j$};
\end{tikzpicture}
\]
\end{enumerate} \begin{enumerate}[resume, label=(2.\arabic*)]
\item\label{rel6} On the other hand, in the case where $j= i-1$, we have the following.
\[
\scalefont{0.8}\begin{tikzpicture}[very thick,xscale=1,yscale=0.5,baseline]
\draw (1,-1) to[in=-90,out=90]  node[below, at start]{$i$} (1.5,0) to[in=-90,out=90] (1,1);
  \draw[dashed] (1.5,-1) to[in=-90,out=90] (1,0) to[in=-90,out=90] (1.5,1);
  \draw (2.5,-1) to[in=-90,out=90]  node[below, at start]{$\tiny i\!-\!1$} (2,0) to[in=-90,out=90] (2.5,1);
\node at (3,0) {=};
  \draw (3.7,-1) -- (3.7,1) node[below, at start]{$i$};
  \draw[dashed] (4.2,-1) to (4.2,1);
  \draw (5.2,-1) -- (5.2,1) node[below, at start]{$\tiny i\!-\!1$} node[midway,fill,inner sep=2.5pt,circle]{};
\node at (5.75,0) {$-$};

  \draw (6.2,-1) -- (6.2,1) node[below, at start]{$i$} node[midway,fill,inner sep=2.5pt,circle]{};
  \draw[dashed] (6.7,-1)-- (6.7,1);
  \draw (7.7,-1) -- (7.7,1) node[below, at start]{$\tiny i\!-\!1$};
\end{tikzpicture}
\]

\item\label{rel7} We also have the relation below, obtained by symmetry.
\[
\scalefont{0.8}\begin{tikzpicture}[very thick,xscale=1,yscale=0.5,baseline]
 \draw[dashed] (1,-1) to[in=-90,out=90] (1.5,0) to[in=-90,out=90] (1,1);
  \draw (1.5,-1) to[in=-90,out=90] node[below, at start]{$i$} (1,0) to[in=-90,out=90] (1.5,1);
  \draw (2,-1) to[in=-90,out=90]  node[below, at start]{$\tiny i\!-\!1$} (2.5,0) to[in=-90,out=90] (2,1);
\node at (3,0) {=};
  \draw[dashed] (3.7,-1) -- (3.7,1);
  \draw (4.2,-1) -- (4.2,1) node[below, at start]{$i$};
  \draw (5.2,-1) -- (5.2,1) node[below, at start]{$\tiny i\!-\!1$} node[midway,fill,inner sep=2.5pt,circle]{};
\node at (5.75,0) {$-$};

  \draw[dashed] (6.2,-1) -- (6.2,1);
  \draw (6.7,-1)-- (6.7,1) node[midway,fill,inner sep=2.5pt,circle]{} node[below, at start]{$i$};
  \draw (7.7,-1) -- (7.7,1) node[below, at start]{$\tiny i\!-\!1$};
\end{tikzpicture}
\]
\end{enumerate}
\begin{enumerate}[resume, label=(2.\arabic*)]
\item\label{rel8} Strands can move through crossings of solid strands freely.
\[
\scalefont{0.8}\begin{tikzpicture}[very thick,scale=0.5 ,baseline]
\draw (-2,-1) -- +(-2,2) node[below,at start]{$k$};
\draw (-4,-1) -- +(2,2) node[below,at start]{$i$};
\draw (-3,-1) .. controls (-4,0) ..  +(0,2)
node[below,at start]{$j$};
\node at (-1,0) {=};
\draw (2,-1) -- +(-2,2)
node[below,at start]{$k$};
\draw (0,-1) -- +(2,2)
node[below,at start]{$i$};
\draw (1,-1) .. controls (2,0) ..  +(0,2)
node[below,at start]{$j$};
\end{tikzpicture}
\]
\end{enumerate}
\indent Similarly, this holds for triple points involving ghosts, except for the following relations when $j=i-1$.
\begin{enumerate}[resume, label=(2.\arabic*)]  \Item\label{rel9}
\[
\scalefont{0.8}\begin{tikzpicture}[very thick,xscale=1,yscale=0.5,baseline]
\draw[dashed] (-2.6,-1) -- +(-.8,2);
\draw[dashed] (-3.4,-1) -- +(.8,2); 
\draw (-1.1,-1) -- +(-.8,2)
node[below,at start]{$\tiny j$};
\draw (-1.9,-1) -- +(.8,2)
node[below,at start]{$\tiny j$}; 
\draw (-3,-1) .. controls (-3.5,0) ..  +(0,2)
node[below,at start]{$i$};

\node at (-.75,0) {=};

\draw[dashed] (.4,-1) -- +(-.8,2);
\draw[dashed] (-.4,-1) -- +(.8,2);
\draw (1.9,-1) -- +(-.8,2)
node[below,at start]{$\tiny j$};
\draw (1.1,-1) -- +(.8,2)
node[below,at start]{$\tiny j$};
\draw (0,-1) .. controls (.5,0) ..  +(0,2)
node[below,at start]{$i$};

\node at (2.25,0) {$-$};

\draw (4.9,-1) -- +(0,2)
node[below,at start]{$\tiny j$};
\draw (4.1,-1) -- +(0,2)
node[below,at start]{$\tiny j$};
\draw[dashed] (3.4,-1) -- +(0,2);
\draw[dashed] (2.6,-1) -- +(0,2);
\draw (3,-1) -- +(0,2) node[below,at start]{$i$};
\end{tikzpicture}
\]
\Item\label{rel10}
\[
\scalefont{0.8}\begin{tikzpicture}[very thick,xscale=1,yscale=0.5,baseline]
\draw[dashed] (-3,-1) .. controls (-3.5,0) ..  +(0,2);  
\draw (-2.6,-1) -- +(-.8,2)
node[below,at start]{$i$};
\draw (-3.4,-1) -- +(.8,2)
node[below,at start]{$i$};
\draw (-1.5,-1) .. controls (-2,0) ..  +(0,2)
node[below,at start]{$\tiny j$};

\node at (-.75,0) {=};

\draw (0.4,-1) -- +(-.8,2)
node[below,at start]{$i$};
\draw (-.4,-1) -- +(.8,2)
node[below,at start]{$i$};
\draw[dashed] (0,-1) .. controls (.5,0) ..  +(0,2);
\draw (1.5,-1) .. controls (2,0) ..  +(0,2)
node[below,at start]{$\tiny j$};

\node at (2.25,0) {$+$};

\draw (3.4,-1) -- +(0,2)
node[below,at start]{$i$};
\draw (2.6,-1) -- +(0,2)
node[below,at start]{$i$}; 
\draw[dashed] (3,-1) -- +(0,2);
\draw (4.5,-1) -- +(0,2)
node[below,at start]{$\tiny j$};
\end{tikzpicture}
\]
\end{enumerate}
In the diagrams with crossings in \ref{rel9} and \ref{rel10}, we say that the solid (respectively ghost) strand bypasses the crossing of ghost strands (respectively solid strands).
The ghost strands may pass through red strands freely.
For $i\neq j$, the solid $i$-strands may pass through red $j$-strands freely.
If the red and solid strands have the same label, a dot is added to the solid strand when straightening.
Diagrammatically, these relations are given by
\begin{enumerate}[resume, label=(\thesection.\arabic*)] \Item\label{rel11}
\[
\scalefont{0.8}\begin{tikzpicture}[very thick,baseline,scale=0.5]
\draw[wei] (-2,-1) -- +(0,2)
node[below,at start]{$i$};
\draw (-1.2,-1)  .. controls (-2.8,0) ..  +(0,2)
node[below,at start]{$i$};

\node at (-.3,0) {=};

\draw (2.8,-1) -- +(0,2)
node[below,at start]{$i$};
\draw[wei] (1.2,-1) -- +(0,2)
node[below,at start]{$i$};
\fill (2.8,0) circle (3pt);

\draw (6.8,-1)  .. controls (5.2,0) ..  +(0,2)
node[below,at start]{$j$};
\draw[wei] (6,-1) -- +(0,2)
node[below,at start]{$i$};

\node at (7.7,0) {=};

\draw[wei] (9.2,-1) -- +(0,2)
node[below,at start]{$i$};
\draw (10.8,-1) -- +(0,2)
node[below,at start]{$j$};
\end{tikzpicture}
\]
\end{enumerate}
and their mirror images.
All solid crossings and dots can pass through red strands, with a correction term.
\begin{enumerate}[resume, label=(\thesection.\arabic*)]
\Item\label{rel12}
\[
\scalefont{0.8}\begin{tikzpicture}[very thick,baseline,scale=0.5]
\draw (-2,-1) .. controls (-4,0) .. +(-2,2)
node[at start,below]{$i$};
\draw (-4,-1) .. controls (-4,0) .. +(2,2)
node [at start,below]{$j$};
\draw[wei] (-3,-1) -- +(0,2)
node[at start,below]{$k$};

\node at (-1,0) {=};

\draw (2,-1) .. controls (2,0) .. +(-2,2)
node[at start,below]{$i$};
\draw (0,-1) .. controls (2,0) .. +(2,2)
node [at start,below]{$j$};
\draw[wei] (1,-1) -- +(0,2)
node[at start,below]{$k$};

\node at (2.8,0) {$+$};

\draw (7.5,-1) -- +(0,2)
node[at start,below]{$i$};
\draw (5.5,-1) -- +(0,2)
node [at start,below]{$j$};
\draw[wei] (6.5,-1) -- +(0,2)
node[at start,below]{$k$};
\node at (3.8,-.2){$\delta_{i,j,k}$};
\end{tikzpicture}
\]
\Item\label{rel13}
\[
\scalefont{0.8}\begin{tikzpicture}[scale=0.5,very thick,baseline=2cm]
\draw[wei] (-2,2) -- +(0,2);
\draw (-4,2) to[in=-90,out=90] (0,4);
\draw (-3,2) to[in=-150,out=90] (-1,4); 
\node at (1,3) {=};

\draw[wei] (4,2) -- +(0,2);
\draw (2,2) to[in=-90,out=90]  (6,4);
\draw (3,2) to[in=-90,out=30]  (5,4);

\draw (12,2) to[in=-90,out=90]  (8,4);
\draw (11,2) to[in=-90,out=150]  (9,4);

\draw[wei] (10,2) -- +(0,2);

\node at (13,3) {=};

\draw (18,2) to[in=-90,out=90] (14,4);
\draw (17,2) to[in=-30,out=90] (15,4);

\draw[wei] (16,2) -- +(0,2);
\end{tikzpicture}
\]

\Item\label{rel14}
\[
\scalefont{0.8}\begin{tikzpicture}[very thick,baseline,scale=0.5]
  \draw(-3,0) +(-1,-1) -- +(1,1);
  \draw[wei](-3,-1) -- (-3,1);
\fill (-3.5,-.5) circle (3pt);
\node at (-1,0) {=};
 \draw(1,0) +(-1,-1) -- +(1,1);
  \draw[wei](1,-1) -- (1,1);
\fill (1.5,.5) circle (3pt);

\draw[wei](5,-1) -- (5,1);
  \draw(1,0) +(5,-1) -- +(3,1);
\fill (5.5,-.5) circle (3pt);
\node at (7,0) {=};
 \draw[wei] (9,-1) -- (9,1); 
  \draw(5,0) +(5,-1) -- +(3,1);
\fill (8.5,.5) circle (3pt);
\end{tikzpicture}
\]
\end{enumerate}
Finally, we have the following non-local idempotent relation.
\begin{enumerate}[resume, label=(\thesection.\arabic*)]
\item\label{rel15}
Any idempotent in which a solid strand is $\ell n$ units to the left of the leftmost red strand is referred to as unsteady and set to be equal to zero.
Using relation \ref{rel1}, this is easily seen to be equivalent to the unsteady relation in \cite{bcs15,bs15,Webster}.
\end{enumerate}
\end{defn}

 The product $d_1 d_2$ of two diagrams $d_1,d_2 \in A(n, \theta, \kappa)$ is given by putting $d_1$ on top of $d_2$.
This product is defined to be $0$ unless the southern border of $d_1$ is given by the same loading as the northern border of $d_2$ with residues of strands matching in the obvious manner.

\begin{rmk}
The relations \ref{rel5}--\ref{rel7}, \ref{rel9}, and \ref{rel10} actually depict an entire region of width $\ell$ inside a given diagram (as we have drawn a solid strand and its ghost).
This has been done to highlight the residue of the ghost strand and the number of dots on the solid strand.
However, the local neighbourhood of interest is that containing the ghost $j$-strand and solid $i$-strand(s); 
the solid $j$-strands need not be in this local neighbourhood.
In particular, we can apply relations \ref{rel5}--\ref{rel7}, \ref{rel9}, and \ref{rel10} even if there is another strand lying between the solid $j$-strand and its ghost.
\end{rmk}

 The diagrammatic Cherednik  algebra is graded as follows:
\begin{itemize}
\item dots have degree 2;
\item the crossing of two strands has degree 0, unless they have the same label, in which case it has degree $-2$;
\item the crossing of a solid strand with label $i$ and a ghost has degree 1 if the ghost has label $i-1$ and 0 otherwise;
\item the crossing of a solid strand with a red strand has degree 0, unless they have the same label, in which case it has degree 1.
\end{itemize}
In other words,
\[
\deg\tikz[baseline,very thick,scale=1.5]{\draw
  (0,.3) -- (0,-.1) node[at end,below,scale=.8]{$i$}
  node[midway,circle,fill=black,inner
  sep=2pt]{};}=
  2 \qquad \deg\tikz[baseline,very thick,scale=1.5]
  {\draw (.2,.3) --
    (-.2,-.1) node[at end,below, scale=.8]{$i$}; \draw
    (.2,-.1) -- (-.2,.3) node[at start,below,scale=.8]{$j$};} =-2\delta_{i,j} \qquad  
  \deg\tikz[baseline,very thick,scale=1.5]{\draw[densely dashed] 
  (-.2,-.1)-- (.2,.3) node[at start,below, scale=.8]{$i$}; \draw
  (.2,-.1) -- (-.2,.3) node[at start,below,scale=.8]{$j$};} =\delta_{j,i+1} \qquad \deg\tikz[baseline,very thick,scale=1.5]{\draw (.2,.3) --
  (-.2,-.1) node[at end,below, scale=.8]{$i$}; \draw [densely dashed]
  (.2,-.1) -- (-.2,.3) node[at start,below,scale=.8]{$j$};} =\delta_{j,i-1}
\]
\[
\deg\tikz[baseline,very thick,scale=1.5]
{
\draw[wei] (0,-.1)-- (0,.3) node[at start,below, scale=.8]{$i$};
  \draw (.2,-.1)  node[below,scale=.8]{$j$};
   \draw (.2,-.1)to [in=-90,out=90]  (-.2,.3)   ;
  }
 = \delta_{i,j} 
\qquad
\deg\begin{tikzpicture}[baseline,very thick,scale=1.5]
 \draw[wei] (0,-.1) -- (0,.3) node[at start,below,scale=.8]{$j$};
   \draw (-.2,-.1)  node[below,scale=.8]{$i$};
   \draw (-.2,-.1)to [in=-90,out=90]  (.2,.3) ;
 \end{tikzpicture} =\delta_{j,i}.
\]

\subsection{A cellular basis of the diagrammatic Cherednik algebra}

Given any $\SSTT \in \SStd(\la,\mu)$, we have a $\theta$-diagram $C_\SSTT$ consisting of a frame with northern and southern distinguished points given by $\mathbf{I}_\mu^\theta$ and $\mathbf{I}_\la^\theta$, respectively, in which the $n$ solid strands each connecting a northern and southern point are drawn so that they trace out the bijection determined by $\SSTT$ in such a way that we use the minimal number of crossings without creating any bigons between pairs of strands or strands and ghosts.
This diagram is not unique up to isotopy (since we have not specified how to resolve triple points), but we can choose one such diagram arbitrarily.

Given a pair of semistandard tableaux of the same shape $(\SSTS,\SSTT)\in\SStd(\la,\mu)\times\SStd(\la,\nu)$, we have a diagram $C_{\SSTS \SSTT}=C_\SSTS C^\ast_\SSTT$ where $C^\ast_\SSTT$ is the diagram obtained from $C_\SSTT $ by flipping it through the horizontal axis.
Notice that there is a unique element $\SSTT^\la \in \SStd(\la,\la)$ and the corresponding basis element $1_\la = C_{\SSTT^\la,\SSTT^\la}$ is the idempotent in which all solid strands are vertical.
A degree function on tableaux is defined in \cite[Defintion 2.13]{Webster};
for our purposes it is enough to note that $\deg(\SSTT)=\deg(C_\SSTT)$ as we shall always work with the $\theta$-diagrams directly.
 The following theorem (see \cite[Sections 2.6 and 4.4]{Webster} for $R$ a field and \cite{bowman17} for $R$ an integral domain) provides a cellular basis, $\mathcal{C}$, of the algebra $A(n,\theta,\kappa)$.
 
\begin{thm}
\label{cellularitybreedscontempt}
The algebra $A(n,\theta,\kappa)$ is a graded cellular algebra with a highest weight theory.
The cellular basis is given by
\[
\mathcal{C}=\{C_{\SSTS \SSTT} \mid \SSTS \in \SStd(\la,\mu), \SSTT\in \SStd(\la,\nu), 
\la,\mu, \nu \in \mptn {\ell}n\}
\]
with respect to the $\theta$-dominance order on the set $\mptn {\ell}n$ and the anti-isomorphism given by flipping a diagram through the horizontal axis.
\end{thm}

\begin{rmk}
We say that $\theta\in\ZZ^\ell$ is a {\sf well-separated weighting} for $A(n,\theta,\kappa)$ if $|\theta_j - \theta_i| > n\ell$ for all $1\leq i < j \leq \ell$.
In \cite[Theorem 3.9]{Webster2} it is shown that if $\theta\in\ZZ^\ell$ is well-separated, then the algebra $A(n,\theta,\kappa)$ is Morita equivalent to the corresponding cyclotomic $q$-Schur algebra.
\end{rmk}

\section{Diagonal cuts of multipartitions}

In this section, we identify subquotients of the diagrammatic Cherednik algebras which are isomorphic to the  product of subquotients of two smaller diagrammatic Cherednik algebras.
We first do this on the level of the indexing sets of representations (in \cref{sdafjhklfsadhjlkafs,sdafjhklfsadhjlkafs2}) and then lift this to the level of algebras (\cref{sdafjhklfsadhjlkafs23}).
We show that the graded decomposition numbers and higher extension groups are preserved under taking these subquotients and hence prove the main theorem stated in the introduction.

\subsection{Diagonal cuts and subquotients of the set of multipartitions}\label{sdafjhklfsadhjlkafs}
Fix $\theta \in \bbz^\ell$ and let $a \in \RR$ be such that the interval $(a-n\varepsilon,a+n\varepsilon)\subset \RR$ does not contain the $x$-coordinate of the loading of a node belonging to any $\ell$-multipartition of $n$.
Given $\la \in \mptn \ell n$, we define the set
\begin{equation}\label{cutineq}
I_a(\la)  = \{(r,c,m) \in [\la]_\theta  \mid a-\ell < \theta_m+\ell(r-c)    <a \}.
\end{equation}
Note that the set $I_a(\la) $ contains at most a single diagonal of boxes from each component of $\la$.  
Similarly, we define
\[
L_a(\la) = \{(r,c,m) \in [\la]_\theta \mid \theta_m+\ell(r-c) < a-\ell\}
\quad R_a(\la) = \{(r,c,m) \in [\la]_\theta \mid  a< \theta_m+\ell(r-c)\}.  
\]
That is, $L_a(\la)$ (respectively $R_a(\la)$) is the loading corresponding to all nodes in $[\la]_\theta$ which are strictly to the left (respectively right) of the line $x=a-\ell$ (respectively $x=a$).
 
\begin{defn}\label{cutdef}
Let $\theta$ be a weighting, $\la,\mu \in \mptn \ell n$ and $a \in \RR$.
We say that the pair $(\la, \mu)$  {\sf admits a $\theta$-diagonal cut at} $x=a$ if the nodes (whose top vertices are) in the region $(a-\ell, a)$ are common to both $\la$ and $\mu$, and the number of nodes strictly to the left or right of the line $x = a$ is the same in both multipartitions; in other words,
\[
I_a(\la) = I_a(\mu), \quad |L_a(\la)| = |L_a(\mu)|, \quad \text{and}  \quad 
|R_a(\la)| = |R_a(\mu)|.
\]
Given $\la \in \mptn \ell n$, we let 
\[
\Lambda_a(\la)=\{\mu   \in \mptn \ell n \mid (\la,\mu)\text{ admits a $\theta$-diagonal cut at $x=a$} \}.
\]
\end{defn}

\begin{rmk}
In \cref{cutdef}, we reference to a region $(a-\ell,a)$.  In our diagrams, we draw the vertical line through the point $x=a$, but we do not draw the vertical line through the point $x=a-\ell$.  Instead, we shade the
  boxes whose top vertices are within the region.  The important point is that the  boxes shaded in  $\lambda$ are also are common to $ \la^R$ and $\la^L$ and that the boxes 
  shaded in  $\mu$ are also are common to $ \mu^R$ and $\mu^L$.   
\end{rmk}

\begin{eg}\label{admitcuteg}
Let $e=3$, $\theta=0$ and $(\la,\mu)=((5,4,3,2,1),(4^3,1^3))$.
This pair admits a diagonal cut at $x = 1/2$.
This cut is slightly to the right of the top vertex of the box $(3,3)$, as can be seen in \cref{anex}.
In this case
\[
I_a(\la)=I_a(\mu)
= \{
(1,1,1),(2,2,1),(3,3,1)
\}
\]
and 
\[
|L_a(\la)|=|L_a(\mu)|=6,
\quad
|R_a(\la)|=|R_a(\mu)|=6.
\]
Note that the pair $(\la,\mu)$ also admits a horizontal cut after the third row, and a vertical cut after the third column (in the sense of \cite{donkinnote,fl03}). That is,
\begin{align*}
5+4+3 = \la_1+\la_2+\la_3 &=  \mu_1+\mu_2+\mu_3=  4+4+4,\\
5+4+3 = \la'_1+\la'_2+\la'_3 &=  \mu'_1+\mu'_2 +\mu'_3 =  6+3+3.
\end{align*}

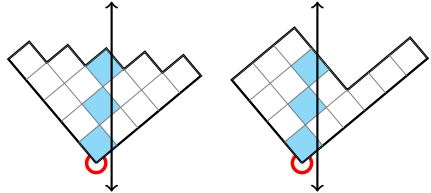
\begin{figure}[ht!]
\begin{center}\scalefont{0.6}
\begin{tikzpicture}[scale=1.8]
  \begin{scope}
{   \path (0,-1.3)++(40:0.2*0.5)++(130:0.2*0.5) coordinate (origin);  
   \draw[wei2]  (0,-1.3)++(40:0.2*0.5)++(130:0.2*0.5)  circle (2pt);
  \draw[thick] (origin) 
   --++(130:5*0.2)
   --++(40:1*0.2)
  --++(-50:1*0.2)	
  --++(40:1*0.2)
  --++(-50:1*0.2)
  --++(40:1*0.2)
  --++(-50:1*0.2)
  --++(40:1*0.2)
  --++(-50:1*0.2)
  --++(40:1*0.2)
  --++(-50:1*0.2)
     --++(220:5*0.2);
     \clip (origin) 
       --++(130:5*0.2)
   --++(40:1*0.2)
  --++(-50:1*0.2)	
  --++(40:1*0.2)
  --++(-50:1*0.2)
  --++(40:1*0.2)
  --++(-50:1*0.2)
  --++(40:1*0.2)
  --++(-50:1*0.2)
  --++(40:1*0.2)
  --++(-50:1*0.2)
     --++(220:5*0.2);
     \draw[fill,cyan!40] (origin) 
    --++(130:1*0.2)
    --++(40:1*0.2)
   --++(130:1*0.2)
    --++(40:1*0.2)
 --++(130:1*0.2)
    --++(40:1*0.2)
--++(-50:1*0.2)
    --++(220:1*0.2)--++(-50:1*0.2)
    --++(220:1*0.2)--++(-50:1*0.2)
    --++(220:1*0.2);  \draw[thick] (origin) 
   --++(130:5*0.2)
   --++(40:1*0.2)
  --++(-50:1*0.2)	
  --++(40:1*0.2)
  --++(-50:1*0.2)
  --++(40:1*0.2)
  --++(-50:1*0.2)
  --++(40:1*0.2)
  --++(-50:1*0.2)
  --++(40:1*0.2)
  --++(-50:1*0.2)
     --++(220:5*0.2);
   \path (40:1cm) coordinate (A1);
  \path (40:2cm) coordinate (A2);
  \path (40:3cm) coordinate (A3);
  \path (40:4cm) coordinate (A4);
  \path (130:1cm) coordinate (B1);
  \path (130:2cm) coordinate (B2);
  \path (130:3cm) coordinate (B3);
  \path (130:4cm) coordinate (B4);
  \path (A1) ++(130:3cm) coordinate (C1);
  \path (A2) ++(130:2cm) coordinate (C2);
  \path (A3) ++(130:1cm) coordinate (C3);
  \foreach \i in {1,...,19}
  {
    \path (origin)++(40:0.2*\i cm)  coordinate (a\i);
    \path (origin)++(130:0.2*\i cm)  coordinate (b\i);
    \path (a\i)++(130:4cm) coordinate (ca\i);
    \path (b\i)++(40:4cm) coordinate (cb\i);
    \draw[thin,gray] (a\i) -- (ca\i)  (b\i) -- (cb\i); } 
  }  \end{scope}

  \path (origin)++(40:0.2*0.4 cm)++(-50:0.2*0.4 cm)++(-90:0.2cm)
 coordinate (XXXX);
 \path (XXXX)++(90:1.4cm) coordinate (XXXXY);
 
   \draw[<->,thick,black] 
    (XXXX)
    --
    (XXXXY);
\end{tikzpicture}        
\quad
\begin{tikzpicture}[scale=1.8]
  \begin{scope}
{   \path (0,-1.3)++(40:0.2*0.5)++(130:0.2*0.5) coordinate (origin);  
   \draw[wei2]  (0,-1.3)++(40:0.2*0.5)++(130:0.2*0.5)  circle (2pt);
  \draw[thick] (origin) 
   --++(130:4*0.2)
   --++(40:3*0.2)
  --++(-50:3*0.2)	
  --++(40:3*0.2)
  --++(-50:1*0.2)
      --++(220:6*0.2);
     \clip (origin) 
     --++(130:4*0.2)
   --++(40:3*0.2)
  --++(-50:3*0.2)	
  --++(40:3*0.2)
  --++(-50:1*0.2)
      --++(220:6*0.2);       \draw[fill,cyan!40] (origin) 
    --++(130:1*0.2)
    --++(40:1*0.2)
   --++(130:1*0.2)
    --++(40:1*0.2)
 --++(130:1*0.2)
    --++(40:1*0.2)
--++(-50:1*0.2)
    --++(220:1*0.2)--++(-50:1*0.2)
    --++(220:1*0.2)--++(-50:1*0.2)
    --++(220:1*0.2);
  \draw[thick] (origin) 
   --++(130:4*0.2)
   --++(40:3*0.2)
  --++(-50:3*0.2)	
  --++(40:3*0.2)
  --++(-50:1*0.2)
      --++(220:6*0.2);
   \path (40:1cm) coordinate (A1);
  \path (40:2cm) coordinate (A2);
  \path (40:3cm) coordinate (A3);
  \path (40:4cm) coordinate (A4);
  \path (130:1cm) coordinate (B1);
  \path (130:2cm) coordinate (B2);
  \path (130:3cm) coordinate (B3);
  \path (130:4cm) coordinate (B4);
  \path (A1) ++(130:3cm) coordinate (C1);
  \path (A2) ++(130:2cm) coordinate (C2);
  \path (A3) ++(130:1cm) coordinate (C3);
  \foreach \i in {1,...,19}
  {
    \path (origin)++(40:0.2*\i cm)  coordinate (a\i);
    \path (origin)++(130:0.2*\i cm)  coordinate (b\i);
    \path (a\i)++(130:4cm) coordinate (ca\i);
    \path (b\i)++(40:4cm) coordinate (cb\i);
    \draw[thin,gray] (a\i) -- (ca\i)  (b\i) -- (cb\i); } 
  }  \end{scope}

  \path (origin)++(40:0.2*0.4 cm)++(-50:0.2*0.4 cm)++(-90:0.2cm)
 coordinate (XXXX);
 \path (XXXX)++(90:1.4cm) coordinate (XXXXY);
 
   \draw[<->,thick,black] 
    (XXXX)
    --
    (XXXXY);
\end{tikzpicture}
\quad
\end{center}
\caption{Fix $e=3$ and $\theta = 0$. The partitions $(5,4,3,2,1)$ and $(4^3,1^3)$ admit a diagonal cut at $x = 1/2$.}
\label{anex}
\end{figure}
\end{eg}

\begin{rmk}
Let $(\la, \mu)$ be a pair of partitions such that $\la \trianglerighteq \mu$.  
The above example illustrates that in level 1, the pair $(\la, \mu)$ admits a $\theta$-diagonal cut if and only if they admit a horizontal cut, if and only if they admit a vertical cut (in the sense of \cite{donkinnote,fl03}).
If $\la \ndom \mu$ (and more generally, if we have $\la \ndom_\theta \mu$ for $\la, \mu \in \mptn \ell n$), the corresponding decomposition numbers and extension groups are zero, which is why we are not concerned with this case here.
\end{rmk}

\begin{lem}\label{middleman}
The set $\Lambda_a(\la)$ is closed under the $\theta$-dominance order.
That is, if $\mu, \mu' \in \Lambda_a(\la)$ and $\nu \in \mptn \ell n$ are such that $\mu \doms_\theta \nu \doms_\theta \mu'$, then $\nu \in \Lambda_a(\la)$.
\end{lem}

\begin{proof}
By the definition of the $\theta$-dominance order, the number of nodes to the left of the  point $x=a-\ell$ for each of the three multipartitions is bounded as follows,  
\begin{equation}\label{anequationyes1}
|L_a(\mu)| \geq
|L_a(\nu)| \geq
|L_a(\mu')|.
\end{equation}
Recall (by the definition of $\Lambda_a(\la)$) that $|L_a(\mu')|=|L_a(\la)| = |L_a(\mu)|$; therefore $|L_a(\la)|=|L_a(\nu)|$, as required.
Similarly, we have that the number of nodes to the right of $x = a$ for each of the multipartitions is bounded as follows,
\begin{equation}\label{anequationyes2}
|R_{a}(\mu')| \geq |R_{a}(\nu)| \geq |R_{a}(\mu)|.
\end{equation}
Therefore (as above) we have that  $|R_a(\nu)| = |R_a(\la)|$, as required.
Finally, it is immediate from the definition that $I_a(\mu)
= I_a(\mu')=I_a(\la)$.
By \cref{anequationyes1,anequationyes2} and our assumption that $\mu \trianglerighteq \nu \trianglerighteq \mu'$, we deduce $I_a(\mu)\trianglerighteq I_a(\nu)\trianglerighteq  I_a(\mu')$, therefore $I_a(\nu)=I_a(\la)$ as required.
\end{proof}

Given a subset $Q\subseteq \Lambda$, we let $E(Q)$ (respectively $F(Q)$) denote the saturated (respectively cosaturated) closure of $Q$ in $\Lambda$; that is, the smallest saturated (respectively cosaturated) set containing $Q$.
By \cref{middleman}, we immediately deduce the following corollary.

\begin{cor}
We have that $\Lambda_a(\la) = E(\Lambda_a(\la)) \cap F(\Lambda_a(\la))$.
\end{cor}

\begin{rmk}
In what follows, when the choice of $\la\in \mptn \ell n$, $a\in \RR$, and $\theta\in \ZZ^\ell$ is clear, we shall let $E:= E(\Lambda_a(\la))$ and $F:= F(\Lambda_a(\la))$.
\end{rmk}

\subsection{A bijection between partially ordered sets}\label{sdafjhklfsadhjlkafs2}
 
We now construct the bijection between partially ordered sets which will underly the isomorphism of subquotient algebras in \cref{sdafjhklfsadhjlkafs23}.

\begin{defn} Given $\la \in \mptn \ell n$,  we define the left-hand side of $\la$ (with respect to a $\theta$-diagonal cut at $x=a$), denoted  $\la^L$, to be the smallest $\ell$-multipartition such that
\[
I_a(\la) \cup L_a(\la)\subseteq [\la^L]_\theta.
\]
Similarly, we define the right-hand side of $\la$, denoted  $\la^R$, 
to be the smallest $\ell$-multipartition such that  
\[
I_a(\la) \cup R_a(\la)\subseteq [\la^L]_\theta. 
\]
We let  $n_L = |\la^L|$ and $n_R = |\la^R|$.  We extend the notation of \cref{cutdef} as follows,
\[
\Lambda^L_a(\la^L)=  \{\mu   \in \mptn \ell {n_L} \mid (\la^L,\mu)\text{ admit a $\theta$-diagonal cut at }x=a \}
\]
and 
\[
\Lambda^R_a(\la^R)=  \{\mu   \in \mptn \ell {n_R} \mid (\la^R,\mu)\text{ admit a $\theta$-diagonal cut at }x=a \}. 
\]
\end{defn}

Informally, we may think of the line $x=a$ as determining the left- and right-hand sides of a multipartition in the obvious way.  
The catch is that these pieces won't be multipartitions, and so we add in as few nodes as possible to make them multipartitions.

\begin{prop}\label{indexbij}
Given $\la \in \mptn \ell n$, we have a bijection
\begin{align*}
\Lambda_a(\la) &\overset{\sim}{\longrightarrow} \Lambda^L_a(\la^{L}) \times\Lambda^R_a(\la^{R}) \\
\mu &\longmapsto (\mu^L, \mu^R).
\end{align*}
\end{prop}

\begin{proof}
This is clear from the definitions.
\end{proof}

\begin{rmk}\label{explicitpieces}
Given $\la \in \mptn \ell n$ and  $\mu \in \Lambda_a (\la)$, one can write down the $\ell$-multipartitions $\mu^L$ and $\mu^R$ explicitly as follows.
Suppose $\mu^{(m)} = (\mu_1, \mu_2, \dots)$.
If every node in $[\mu^{(m)}]_\theta$ lies to the left of $a-\ell$ (respectively right of $a$), then the $m$th component of $\mu^L$ (respectively $\mu^{R}$) is just $\mu^{(m)}$ itself.

Otherwise, at least one node in $[\mu^{(m)}]_\theta$ has $x$-coordinate in $(a-\ell, a)$.
Suppose the highest such node is $(r,c,m) \in [\mu^{(m)}]_\theta$.
Then the $m$th components of $\mu^L$  and $\mu^R$ are
\[
(\mu^L)^{(m)}=(\mu_1, \mu_2,\dots, \mu_r) \quad \text{and} \quad  (\mu^R)^{(m)}=(c^r,\mu_{r+1}, \mu_{r+2},\dots).
\]
\end{rmk}

\begin{eg}
Continuing \cref{admitcuteg}, recall that $e=3$, $\theta=0$, $\la = (5,4,3,2,1)$ and $\mu = (4^3,1^3)$. Making a diagonal cut at $x = 1/2$ yields
\[
\la^L = (5,4,3), \quad \la^R = (3^3,2,1), \quad \mu^L = (4^3), \quad \mu^R = (3^3,1^3).
\]
These are depicted in \cref{cuteg}. We note that the only difference between our diagonal cuts and the more classical horizontal and vertical cuts is that we have an `extra rectangle' of boxes in the form of a partition $(r^c)$ common to {\em both} $\la^L$ and $\la^R$.
\begin{figure}[ht!] 
\begin{tikzpicture}[scale=1.5]
  \begin{scope}
{   \path (0,-1.3)++(40:0.2*0.5)++(130:0.2*0.5) coordinate (origin);  
   \draw[wei2]  (0,-1.3)++(40:0.2*0.5)++(130:0.2*0.5)  circle (2pt);
  \draw[thick] (origin) 
   --++(130:5*0.2)
   --++(40:1*0.2)
  --++(-50:1*0.2)	
  --++(40:1*0.2)
  --++(-50:1*0.2)
  --++(40:1*0.2)
  --++(-50:3*0.2)
     --++(220:3*0.2);
     \clip (origin) 
       --++(130:5*0.2)
   --++(40:1*0.2)
  --++(-50:1*0.2)	
  --++(40:1*0.2)
  --++(-50:1*0.2)
  --++(40:1*0.2)
  --++(-50:3*0.2)
     --++(220:3*0.2);      \draw[fill,cyan!40] (origin) 
    --++(130:1*0.2)
    --++(40:1*0.2)
   --++(130:1*0.2)
    --++(40:1*0.2)
 --++(130:1*0.2)
    --++(40:1*0.2)
--++(-50:1*0.2)
    --++(220:1*0.2)--++(-50:1*0.2)
    --++(220:1*0.2)--++(-50:1*0.2)
    --++(220:1*0.2);
   \path (40:1cm) coordinate (A1);
  \path (40:2cm) coordinate (A2);
  \path (40:3cm) coordinate (A3);
  \path (40:4cm) coordinate (A4);
  \path (130:1cm) coordinate (B1);
  \path (130:2cm) coordinate (B2);
  \path (130:3cm) coordinate (B3);
  \path (130:4cm) coordinate (B4);
  \path (A1) ++(130:3cm) coordinate (C1);
  \path (A2) ++(130:2cm) coordinate (C2);
  \path (A3) ++(130:1cm) coordinate (C3);
  \foreach \i in {1,...,19}
  {
    \path (origin)++(40:0.2*\i cm)  coordinate (a\i);
    \path (origin)++(130:0.2*\i cm)  coordinate (b\i);
    \path (a\i)++(130:4cm) coordinate (ca\i);
    \path (b\i)++(40:4cm) coordinate (cb\i);
    \draw[thin,gray] (a\i) -- (ca\i)  (b\i) -- (cb\i); } 
  }  \end{scope}

  \path (origin)++(40:0.2*0.4 cm)++(-50:0.2*0.4 cm)++(-90:0.2cm)
 coordinate (XXXX);
 \path (XXXX)++(90:1.4cm) coordinate (XXXXY);
 
   \draw[<->,thick,black] 
    (XXXX)
    --
    (XXXXY);
\end{tikzpicture}
\quad
\begin{tikzpicture}[scale=1.5]
  \begin{scope}
{   \path (0,-1.3)++(40:0.2*0.5)++(130:0.2*0.5) coordinate (origin);  
   \draw[wei2]  (0,-1.3)++(40:0.2*0.5)++(130:0.2*0.5)  circle (2pt);
  \draw[thick] (origin) 
   --++(130:4*0.2)
   --++(40:3*0.2)
  --++(-50:4*0.2)	
      --++(220:3*0.2);
     \clip (origin) 
     --++(130:4*0.2)
   --++(40:3*0.2)
  --++(-50:4*0.2)	
      --++(220:3*0.2);      \draw[fill,cyan!40] (origin) 
    --++(130:1*0.2)
    --++(40:1*0.2)
   --++(130:1*0.2)
    --++(40:1*0.2)
 --++(130:1*0.2)
    --++(40:1*0.2)
--++(-50:1*0.2)
    --++(220:1*0.2)--++(-50:1*0.2)
    --++(220:1*0.2)--++(-50:1*0.2)
    --++(220:1*0.2);

   \path (40:1cm) coordinate (A1);
  \path (40:2cm) coordinate (A2);
  \path (40:3cm) coordinate (A3);
  \path (40:4cm) coordinate (A4);
  \path (130:1cm) coordinate (B1);
  \path (130:2cm) coordinate (B2);
  \path (130:3cm) coordinate (B3);
  \path (130:4cm) coordinate (B4);
  \path (A1) ++(130:3cm) coordinate (C1);
  \path (A2) ++(130:2cm) coordinate (C2);
  \path (A3) ++(130:1cm) coordinate (C3);
  \foreach \i in {1,...,19}
  {
    \path (origin)++(40:0.2*\i cm)  coordinate (a\i);
    \path (origin)++(130:0.2*\i cm)  coordinate (b\i);
    \path (a\i)++(130:4cm) coordinate (ca\i);
    \path (b\i)++(40:4cm) coordinate (cb\i);
    \draw[thin,gray] (a\i) -- (ca\i)  (b\i) -- (cb\i); } 
  }  \end{scope}

  \path (origin)++(40:0.2*0.4 cm)++(-50:0.2*0.4 cm)++(-90:0.2cm)
 coordinate (XXXX);
 \path (XXXX)++(90:1.4cm) coordinate (XXXXY);
 
   \draw[<->,thick,black] 
    (XXXX)
    --
    (XXXXY);
\end{tikzpicture}
\quad \quad \quad \quad
\begin{tikzpicture}[scale=1.5]
\begin{scope}
{   \path (0,-1.3)++(40:0.2*0.5)++(130:0.2*0.5) coordinate (origin);  
   \draw[wei2]  (0,-1.3)++(40:0.2*0.5)++(130:0.2*0.5)  circle (2pt);
  \draw[thick] (origin) 
   --++(130:3*0.2)
   --++(40:3*0.2)
  --++(-50:1*0.2)
  --++(40:1*0.2)
  --++(-50:1*0.2)
  --++(40:1*0.2)
  --++(-50:1*0.2)
     --++(220:5*0.2);
     \clip (origin) 
     --++(130:3*0.2)
   --++(40:3*0.2)
  --++(-50:1*0.2)
  --++(40:1*0.2)
  --++(-50:1*0.2)
  --++(40:1*0.2)
  --++(-50:1*0.2)
     --++(220:5*0.2);      \draw[fill,cyan!40] (origin) 
    --++(130:1*0.2)
    --++(40:1*0.2)
   --++(130:1*0.2)
    --++(40:1*0.2)
 --++(130:1*0.2)
    --++(40:1*0.2)
--++(-50:1*0.2)
    --++(220:1*0.2)--++(-50:1*0.2)
    --++(220:1*0.2)--++(-50:1*0.2)
    --++(220:1*0.2);
   \path (40:1cm) coordinate (A1);
  \path (40:2cm) coordinate (A2);
  \path (40:3cm) coordinate (A3);
  \path (40:4cm) coordinate (A4);
  \path (130:1cm) coordinate (B1);
  \path (130:2cm) coordinate (B2);
  \path (130:3cm) coordinate (B3);
  \path (130:4cm) coordinate (B4);
  \path (A1) ++(130:3cm) coordinate (C1);
  \path (A2) ++(130:2cm) coordinate (C2);
  \path (A3) ++(130:1cm) coordinate (C3);
  \foreach \i in {1,...,19}
  {
    \path (origin)++(40:0.2*\i cm)  coordinate (a\i);
    \path (origin)++(130:0.2*\i cm)  coordinate (b\i);
    \path (a\i)++(130:4cm) coordinate (ca\i);
    \path (b\i)++(40:4cm) coordinate (cb\i);
    \draw[thin,gray] (a\i) -- (ca\i)  (b\i) -- (cb\i); } 
  }  \end{scope}

  \path (origin)++(40:0.2*0.4 cm)++(-50:0.2*0.4 cm)++(-90:0.2cm)
 coordinate (XXXX);
 \path (XXXX)++(90:1.4cm) coordinate (XXXXY);
 
   \draw[<->,thick,black] 
    (XXXX)
    --
    (XXXXY);
\end{tikzpicture}
\quad
\begin{tikzpicture}[scale=1.5]
\begin{scope}
{   \path (0,-1.3)++(40:0.2*0.5)++(130:0.2*0.5) coordinate (origin);  
   \draw[wei2]  (0,-1.3)++(40:0.2*0.5)++(130:0.2*0.5)  circle (2pt);
  \draw[thick] (origin) 
   --++(130:3*0.2)
   --++(40:3*0.2)
  --++(-50:2*0.2)	
  --++(40:3*0.2)
  --++(-50:1*0.2)
      --++(220:6*0.2);
     \clip (origin) 
     --++(130:3*0.2)
   --++(40:3*0.2)
  --++(-50:2*0.2)	
  --++(40:3*0.2)
  --++(-50:1*0.2)
      --++(220:6*0.2);     \draw[fill,cyan!40] (origin) 
    --++(130:1*0.2)
    --++(40:1*0.2)
   --++(130:1*0.2)
    --++(40:1*0.2)
 --++(130:1*0.2)
    --++(40:1*0.2)
--++(-50:1*0.2)
    --++(220:1*0.2)--++(-50:1*0.2)
    --++(220:1*0.2)--++(-50:1*0.2)
    --++(220:1*0.2);

   \path (40:1cm) coordinate (A1);
  \path (40:2cm) coordinate (A2);
  \path (40:3cm) coordinate (A3);
  \path (40:4cm) coordinate (A4);
  \path (130:1cm) coordinate (B1);
  \path (130:2cm) coordinate (B2);
  \path (130:3cm) coordinate (B3);
  \path (130:4cm) coordinate (B4);
  \path (A1) ++(130:3cm) coordinate (C1);
  \path (A2) ++(130:2cm) coordinate (C2);
  \path (A3) ++(130:1cm) coordinate (C3);
  \foreach \i in {1,...,19}
  {
    \path (origin)++(40:0.2*\i cm)  coordinate (a\i);
    \path (origin)++(130:0.2*\i cm)  coordinate (b\i);
    \path (a\i)++(130:4cm) coordinate (ca\i);
    \path (b\i)++(40:4cm) coordinate (cb\i);
    \draw[thin,gray] (a\i) -- (ca\i)  (b\i) -- (cb\i); } 
  }  \end{scope}

  \path (origin)++(40:0.2*0.4 cm)++(-50:0.2*0.4 cm)++(-90:0.2cm)
 coordinate (XXXX);
 \path (XXXX)++(90:1.4cm) coordinate (XXXXY);
 
   \draw[<->,thick,black] 
    (XXXX)
    --
    (XXXXY);
\end{tikzpicture}
\caption{Let $(\la,\mu)$ be the pair of partitions admitting a diagonal cut as in \cref{anex}. We picture the resulting partitions above. The leftmost diagram depicts  $\la^L,\mu^L$, and the rightmost diagram depicts $\la^R,\mu^R$.}
\label{cuteg}
\end{figure}
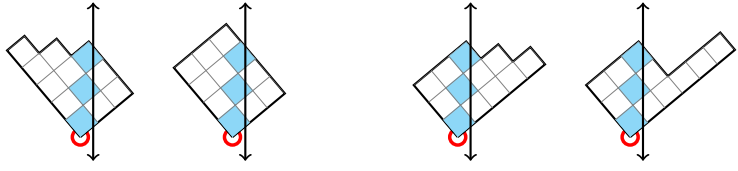
\end{eg}

\subsection{A subquotient of the diagrammatic Cherednik algebra}\label{sdafjhklfsadhjlkafs23}
We now define the subquotients of the diagrammatic Cherednik algebras in which we shall be interested for the remainder of the paper.
Fix $n\in \mathbb{N}$, $\theta\in \ZZ^\ell$, and $\kappa\in (\ZZ/e\ZZ)^\ell$.
We now choose  $\la\in \mptn \ell n$, $a\in \RR$ and we  let $E:= E(\Lambda_a(\la))$ and $F:= F(\Lambda_a(\la))$.
As in \cref{subquot}, we let
\[
e = \sum_{\mu \in E \cap F} 1_\mu \quad \text{and} \quad
f = \sum_{\mu \in F\setminus E} 1_\mu
\]
and we define
\[
A_{\Lambda_a}  =  e(A/(A f A))e
\]
for $A:=A(n,\theta,\kappa)$.
By \cref{3.6rmk} this algebra is a graded cellular algebra with basis
\[
\{C_{\SSTS \SSTT}\mid \SSTS \in \SStd(\alpha,\beta), \, \SSTT\in \SStd(\alpha,\gamma), \, \alpha, \beta, \gamma \in 
\Lambda_a (\la)
\}.
\]
If $R$ is a field, then $A_{\Lambda_a} $ is quasi-hereditary.

\begin{defn}
Let $\mu, \nu \in \Lambda_a(\la)$.
For $\SSTS \in \SStd(\mu,\nu)$, we define $\SSTS^L$ and $\SSTS^R$ to be the tableaux determined by
\begin{enumerate}
[leftmargin=*,itemsep=0.5em]\item $\SSTS^L(r,c,k) = \SSTS(r,c,k)$ for $\mathbf{I}_{(r,c,k)}^\theta < a$ (respectively $\SSTS^R(r,c,k) =\SSTS(r,c,k)$ for $\mathbf{I}_{(r,c,k)}^\theta > a$);
\item $\SSTS^L(r,c,k) =  \mathbf{I}_{(r,c,k)}^\theta$ for  $\mathbf{I}_{(r,c,k)}^\theta > a$ (respectively  $\SSTS^R(r,c,k) =  \mathbf{I}_{(r,c,k)}^\theta$ for  $\mathbf{I}_{(r,c,k)}^\theta < a$);
\end{enumerate}
for $(r,c,k) \in [\mu^L]_\theta$ (respectively $(r,c,k) \in [\mu^R]_\theta$).
\end{defn}

\begin{thm}\label{vsisom}
Given $\la\in \mptn \ell n$ and $a \in \RR$.
We have an isomorphism
\[
A_{\Lambda_a(\la)} \cong A_{\Lambda^L_a(\la^L)} \otimes_{\Bbbk} A_{ \Lambda^R_a(\la^R)}
\]
of graded $R$-modules given by the map
\[
\varphi: C_{\SSTS \SSTT} \longmapsto  C_{\SSTS^L \SSTT^L} \otimes C_{\SSTS^R \SSTT^R}.  
\]
\end{thm}

\begin{proof}
Recall from \cref{indexbij} that the map $\Lambda_a(\la) \cong \Lambda^L_a(\la^L) \times \Lambda^R_a(\la^R)$ given by $\mu \mapsto \mu^L \times \mu^R$, is bijective.
Given $\mu,\nu \in \Lambda_a(\la)$, we have that
\begin{align*}
|\mathbf{I}_\mu^\theta \cap (-\infty, a-\ell)| &= |\mathbf{I}_\nu^\theta \cap (-\infty, a-\ell)|,\\
|\mathbf{I}_\mu^\theta \cap (a-\ell, a)| &= |\mathbf{I}_\nu^\theta \cap (a-\ell, a)|,\\
|\mathbf{I}_\mu^\theta \cap (a, \infty)| &= |\mathbf{I}_\nu^\theta \cap (a, \infty)|,
\end{align*}
and therefore (by \cref{hjkljkhjklh}) any tableau $\SSTS\in \SStd(\mu,\nu)$ satisfies
\[
\SSTS(r,c,k) =\mathbf{I}^\theta_{(r',c',k')},
\]
for $ {(r,c,k)} $ and ${(r',c',k')} $ a pair of boxes whose loadings belong the same region: $(-\infty, a-\ell)$, $(a-\ell, a)$, or $(a-\ell, \infty)$ respectively.
Therefore $ \SSTS^L \in \SStd(\mu^L,\nu^L)$, $\SSTS^R \in \SStd(\mu^R,\nu^R)$, and the map $\SSTS \mapsto \SSTS^L \times \SSTS^R$  is a bijection $\SStd(\mu,\nu) \rightarrow \SStd(\mu^L,\nu^L) \times \SStd(\mu^R,\nu^R)$.
Moreover, we have that
\begin{equation}\label{sdfafdsa}
\begin{array}{c}
C_{\SSTS^L} \cap (-\infty, a-\ell) \times [0,1] = 
C_\SSTS \cap (-\infty, a-\ell) \times [0,1]
\\
C_{\SSTS^R} \cap (a, \infty) \times [0,1] = 
C_\SSTS  \cap (a, \infty) \times [0,1].
\end{array}
\end{equation}
Given any two multipartitions $\mu^L,\nu^L \in \Lambda^L_a(n_L,\theta,\kappa)$ (or $\mu^R,\nu^R \in \Lambda^R_a(n_R, \theta, \kappa)$) all nodes to the right of the point $a-\ell$ (respectively to the left of $a$) are common to both multipartitions.
Therefore by the definition of semistandard tableaux, 
we have that
\[
\SSTU(r,c,k) = \mathbf{I}_{(r,c,k)}^\theta \qquad \text{and} \qquad \SSTV(r',c',k') =  \mathbf{I}_{(r',c',k')}^\theta
\]
for $\mathbf{I}_{(r,c,k)}^\theta > a - \ell$, $\mathbf{I}_{(r,c,k)}^\theta < a$
 and  $\SSTU \in \SStd(\mu^L,\nu^L)$,  $\SSTV  \in \SStd(\mu^R, \nu^R)$.
Therefore, we have that
\begin{equation}\label{sdfafdsa2}
\begin{array}{r}
C_{\SSTS^L} \cap ((a-\ell, \infty) \times [0,1] ) = 
1_\mu  \cap ((a-\ell, \infty) \times [0,1])  = 1_\nu \cap ((a-\ell,\infty) \times [0,1])
\\
C_{\SSTS^R} \cap ((-\infty,a) \times [0,1]  )= 
1_\mu  \cap ( (-\infty,a)  \times [0,1])  = 1_\nu  \cap ( (-\infty,a) \times [0,1])
\end{array}
\end{equation}
and the diagrams in \cref{sdfafdsa2} are both of degree zero.
Therefore by \cref{sdfafdsa,sdfafdsa2} we have that
\[
\deg(C_{\SSTS^L})
+
\deg(C_{\SSTS^R})
=
\deg(C_\SSTS  \cap (-\infty, a-\ell) \times [0,1])
+
\deg(C_\SSTS  \cap (a, \infty) \times [0,1])
=
\deg(C_\SSTS)
\]
and the map $C_\SSTS \mapsto C_{\SSTS^L} \times C_{\SSTS^R}$  is a degree preserving bijection on the cell modules.
The result follows.
\end{proof}

\begin{eg}
Recall our running example, with $e=3$, $\theta=0$ and $(\la,\mu)=$ \linebreak $((5,4,3,2,1),(4^3,1^3))$.
There is a unique element $\SSTS \in \SStd(\la,\mu)$ and the corresponding basis element $C_\SSTS$ is depicted in \cref{dsaflfsdajkl}.
This pair admits a $\theta$-diagonal cut at $x=1/2$.
The elements $C_{\SSTS^L}$ and $C_{\SSTS^R}$ are depicted in \cref{dsaflfsdajkl2}.
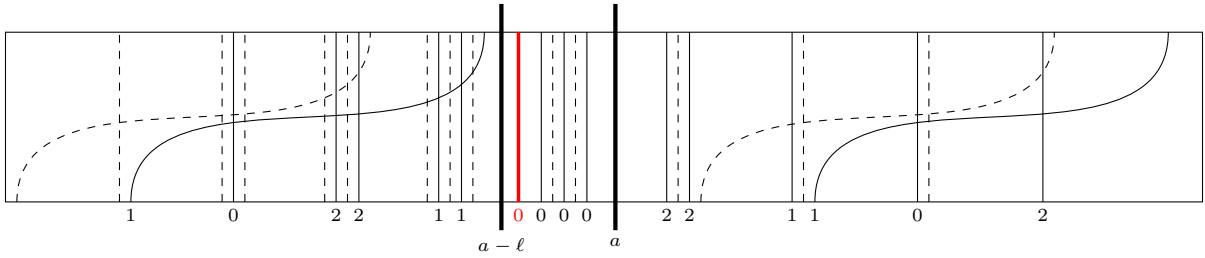
\begin{figure}[ht!]
\[
\scalefont{0.6}
    \begin{tikzpicture}[scale=0.75] 
  \draw (-9,0) rectangle (12,3);
       \draw[ultra thick](1.7,-0.5)--(1.7,3.5);
    \node [below] at (1.7,-0.5) {\tiny $a$};
       \draw[ultra thick](1.7-2,-0.5)--(1.7-2,3.5);
           \node [below] at (1.7-2,-0.5) {\tiny $a-\ell$};
     
 \draw(-6.8,0)  to [out=90,in=-90] (-2+1.4,3);
    \node [below] at (-6.8,0) {\tiny $1$};
 \draw[dashed] (-6.8-2,0)  to [out=90,in=-90] (-2+1.4-2,3);

 \draw(-5,0)--(-5,3);
    \node [below] at (-5,0) {\tiny $0$};
     \draw[dashed](-5-2,0)--(-5-2,3);

 \draw(-3.2,0)--(-3.2,3) (-2.8,0)--(-2.8,3);
    \node [below] at (-3.2,0) {\tiny $2$};
      \node [below] at (-2.8,0) {\tiny $2$};
       \draw[dashed](-3.2-2,0)--(-3.2-2,3) (-2.8-2,0)--(-2.8-2,3);
 \draw(-1.4,0)--(-1.4,3) (-1,0)--(-1,3);
   \node [below] at (-1.4,0) {\tiny $1$};
      \node [below] at (-1,0) {\tiny $1$};
 \draw[dashed] (-1.4-2,0)--(-1.4-2,3) (-1-2,0)--(-1-2,3);

   \draw[wei2] (0,0)--(0,3);
   \draw(0.4,0)--(0.4,3) (0.8,0)--(0.8,3) (1.2,0)--(1.2,3);
    \node [wei2,below] at (0,0) {\tiny $0$};
        \node [below] at (0.8,0) {\tiny $0$}; \node [below] at (0.4,0) {\tiny $0$};
         \node [below] at (1.2,0) {\tiny $0$};
   \draw[dashed](0.4-2,0)--(0.4-2,3) (0.8-2,0)--(0.8-2,3) (1.2-2,0)--(1.2-2,3);

\draw (2.6,0)--(2.6,3)         (3,0)--(3,3);
        \node [below] at (2.6,0) {\tiny $2$}; 
                \node [below] at (3,0) {\tiny $2$}; 
\draw[dashed] (2.6-2,0)--(2.6-2,3)         (3-2,0)--(3-2,3);
\draw (4.8,0)--(4.8,3) ;
\draw        (5.2,0)  to [out=90,in=-90] (11.4,3);
        \node [below] at (4.8,0) {\tiny $1$}; 
                \node [below] at (5.2,0) {\tiny $1$}; 
\draw[dashed] (4.8-2,0)--(4.8-2,3) ;
\draw [dashed]       (5.2-2,0)  to [out=90,in=-90] (11.4-2,3);

  \draw (7,0)--(7,3)        ;
        \node [below] at (7,0) {\tiny $0$}; 
  \draw[dashed] (7-2,0)--(7-2,3);
  \draw (9.2,0)--(9.2,3);
        \node [below] at (9.2,0) {\tiny $2$};
  \draw[dashed] (9.2-2,0)--(9.2-2,3);
\end{tikzpicture}
\]
\caption{The basis element $C_\SSTS$ for $(\mu,\nu)=((5,4,3,2,1),(4^3,1^3))$.
Given $a = 1/2$, we mark the lines of the cut, $x = a$ and $x = a-\ell$.}
\label{dsaflfsdajkl}
\end{figure}
 
\begin{figure}[ht!]
\[
\scalefont{0.6}
\begin{tikzpicture}[scale=0.75] 
  \draw (-9,0) rectangle (12,3);
 \draw(-6.8,0)  to [out=90,in=-90] (-2+1.4,3);
    \node [below] at (-6.8,0) {\tiny $1$};
 \draw[dashed] (-6.8-2,0)  to [out=90,in=-90] (-2+1.4-2,3);

 \draw(-5,0)--(-5,3);
    \node [below] at (-5,0) {\tiny $0$};
     \draw[dashed](-5-2,0)--(-5-2,3);

 \draw(-3.2,0)--(-3.2,3) (-2.8,0)--(-2.8,3);
    \node [below] at (-3.2,0) {\tiny $2$};
      \node [below] at (-2.8,0) {\tiny $2$};
       \draw[dashed](-3.2-2,0)--(-3.2-2,3) (-2.8-2,0)--(-2.8-2,3);
 \draw(-1.4,0)--(-1.4,3) (-1,0)--(-1,3);
   \node [below] at (-1.4,0) {\tiny $1$};
      \node [below] at (-1,0) {\tiny $1$};
 \draw[dashed] (-1.4-2,0)--(-1.4-2,3) (-1-2,0)--(-1-2,3);

   \draw[wei2] (0,0)--(0,3);
   \draw(0.4,0)--(0.4,3) (0.8,0)--(0.8,3) (1.2,0)--(1.2,3);
    \node [wei2,below] at (0,0) {\tiny $0$};
        \node [below] at (0.8,0) {\tiny $0$}; \node [below] at (0.4,0) {\tiny $0$};
         \node [below] at (1.2,0) {\tiny $0$};
   \draw[dashed](0.4-2,0)--(0.4-2,3) (0.8-2,0)--(0.8-2,3) (1.2-2,0)--(1.2-2,3);

\draw (2.6,0)--(2.6,3)         (3,0)--(3,3);
        \node [below] at (2.6,0) {\tiny $2$}; 
                \node [below] at (3,0) {\tiny $2$}; 
\draw[dashed] (2.6-2,0)--(2.6-2,3)         (3-2,0)--(3-2,3);
\draw (4.8,0)--(4.8,3) ;
         \node [below] at (4.8,0) {\tiny $1$}; 
 \draw[dashed] (4.8-2,0)--(4.8-2,3);
\end{tikzpicture}
\]
\[
    \scalefont{0.6}
    \begin{tikzpicture}[scale=0.75] 
  \draw (-9,0) rectangle (12,3);
 
 \draw(-3.2,0)--(-3.2,3); 
    \node [below] at (-3.2,0) {\tiny $2$};
       \draw[dashed](-3.2-2,0)--(-3.2-2,3); 
 \draw(-1.4,0)--(-1.4,3) (-1,0)--(-1,3);
   \node [below] at (-1.4,0) {\tiny $1$};
      \node [below] at (-1,0) {\tiny $1$};
 \draw[dashed] (-1.4-2,0)--(-1.4-2,3) (-1-2,0)--(-1-2,3);

   \draw[wei2] (0,0)--(0,3);
   \draw(0.4,0)--(0.4,3) (0.8,0)--(0.8,3) (1.2,0)--(1.2,3);
    \node [wei2,below] at (0,0) {\tiny $0$};
        \node [below] at (0.8,0) {\tiny $0$}; \node [below] at (0.4,0) {\tiny $0$};
         \node [below] at (1.2,0) {\tiny $0$};
   \draw[dashed](0.4-2,0)--(0.4-2,3) (0.8-2,0)--(0.8-2,3) (1.2-2,0)--(1.2-2,3);

\draw (2.6,0)--(2.6,3)         (3,0)--(3,3);
        \node [below] at (2.6,0) {\tiny $2$}; 
                \node [below] at (3,0) {\tiny $2$}; 
\draw[dashed] (2.6-2,0)--(2.6-2,3)         (3-2,0)--(3-2,3);
\draw (4.8,0)--(4.8,3) ;
\draw        (5.2,0)  to [out=90,in=-90] (11.4,3);
        \node [below] at (4.8,0) {\tiny $1$}; 
                \node [below] at (5.2,0) {\tiny $1$}; 
\draw[dashed] (4.8-2,0)--(4.8-2,3) ;
\draw [dashed]       (5.2-2,0)  to [out=90,in=-90] (11.4-2,3);

  \draw (7,0)--(7,3)        ;
        \node [below] at (7,0) {\tiny $0$}; 
  \draw[dashed] (7-2,0)--(7-2,3)        ;
  \draw (9.2,0)--(9.2,3)        ;
        \node [below] at (9.2,0) {\tiny $2$}; 
  \draw[dashed] (9.2-2,0)--(9.2-2,3)        ;
\end{tikzpicture}
\]
\caption{The basis elements $C_{\SSTS^L}$ and  $C_{\SSTS^R}$ corresponding to $(\mu,\nu)=((5,4,3,2,1),(4^3,1^3))$ cut at $x = 1/2$ and the unique $\SSTS\in\SStd(\mu,\nu)$.}\label{dsaflfsdajkl2}
\end{figure}
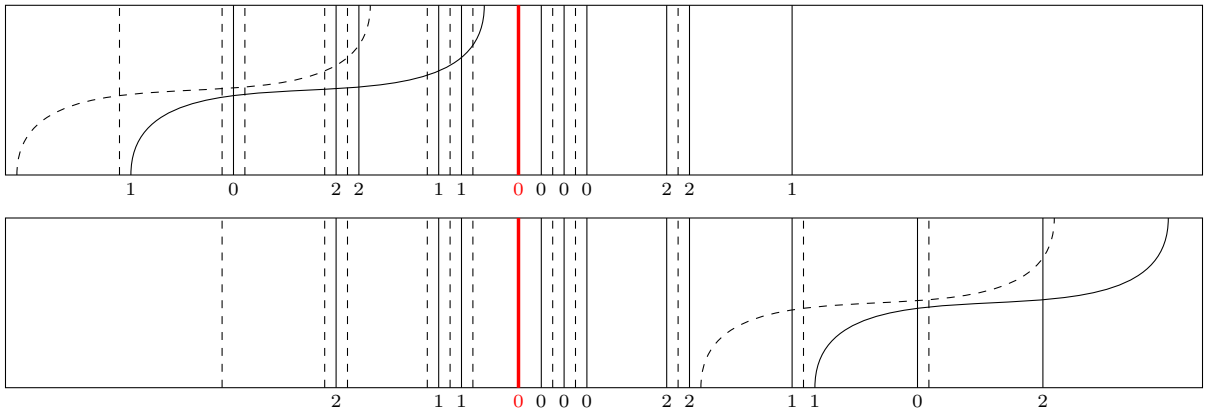
\end{eg}

\begin{thm}\label{algisom}
The map $\varphi: A_{\Lambda_a(\la)} \to  A_{\Lambda^L_a(\la^L)} \otimes_{\Bbbk} A_{\Lambda^R_a(\la^R)}$, defined in \cref{vsisom}, is an isomorphism of graded $R$-algebras.
\end{thm}

\begin{proof}
We have already seen in \cref{vsisom} that the map $\varphi$ is an isomorphism of graded $R$-modules; it remains to check that
\[
C_{\SSTS^L\SSTT^L}C_{\SSTU^L\SSTV^L} \otimes C_{\SSTS^R\SSTT^R}C_{\SSTU^R\SSTV^R} =\varphi(C_{\SSTS\SSTT}) \varphi(C_{\SSTU\SSTV}) =\varphi(C_{\SSTS\SSTT}C_{\SSTU\SSTV}).
\]
The cell-ideals of $A(n,\theta,\kappa)$ are ordered according to the $\theta$-dominance order on loadings; more dominant loadings in this ordering are given by moving strands to the left.
Therefore in order to rewrite a product in terms of the cellular basis, we shall proceed by pulling the strands in the diagrams as far to the left as possible using relations \ref{rel1} to \ref{rel15}.

The relations of $A(n,\theta,\kappa)$ (and therefore those of the subquotients $A_{\Lambda_a(\la)}$, $A_{\Lambda^L_a(\la^L)}$, and  $A_{\Lambda^R_a(\la^R)}$) can all be applied to the regions  $(-\infty, a-\ell)$, then $(a-\ell, a)$, and then $(a, \infty)$ to each of the diagrams in turn.
This is because relations \ref{rel1} to \ref{rel14} of $A(m,\theta,\kappa)$ for $m=n, n_L, n_R$ are local, and the additional idempotent relations (relation \ref{rel15} and the cosaturated quotient relation in the $\theta$-dominance order) are all given in terms of loadings, and therefore are compatible with restriction to regions of $\RR$.
Consider a diagram, $D$, in any one of the algebras $A_{\Lambda_a(\la)}$, $A_{\Lambda^L_a(\la^L)}$, and $A_{ \Lambda^R_a(\la^R)}$.
By \cite[Theorem 4.1]{bowman17}, if we can pull an $i$-strand from $D\cap (a,a+\ell)$ to the left of the line $x=a-\ell$ or a strand from $(a+\ell,\infty)$ to the left of the line $x=a$, then the resulting diagram is zero.

First, apply relations \ref{rel1} to \ref{rel15} to the region $(-\infty, a-\ell)$ of the diagrams $C_{\SSTS\SSTT}C_{\SSTU\SSTV}$ and $C_{\SSTS^L\SSTT^L}C_{\SSTU^L\SSTV^L}$ concurrently to push all strands in both
\[
C_{\SSTS\SSTT}C_{\SSTU\SSTV} \cap ((-\infty, a-\ell)\times [0,1])\quad \text{ and } \quad C_{\SSTS^L\SSTT^L}C_{\SSTU^L\SSTV^L} \cap ((-\infty, a-\ell)\times [0,1])
\]
as far to the left as possible.
Next, apply relations \ref{rel1} to \ref{rel15} to the region $(a , \infty)$ of the diagrams $C_{\SSTS\SSTT}C_{\SSTU\SSTV}$ and $C_{\SSTS^R\SSTT^R}C_{\SSTU^R\SSTV^R}$ concurrently to push all strands in both
\[
C_{\SSTS\SSTT}C_{\SSTU\SSTV} \cap ((a ,\infty)\times [0,1])\quad \text{ and } \quad  C_{\SSTS^R\SSTT^R}C_{\SSTU^R\SSTV^R} \cap ((a ,\infty)\times [0,1])
\]
as far to the left as possible.
If every element in the resulting linear combination of diagrams in $A(n,\theta,\kappa)$ is an element of the basis $\mathcal{C}$, then the result follows.
Otherwise, let $d$ be any diagram in the linear combination which is not in $\mathcal{C}$.
By assumption, the diagram $d$ contains a strand which cannot be pushed further left (using relations \ref{rel1} to \ref{rel15}) in either $(a-\ell, a)$ or $(a, \infty)$, but which can be pushed further left in $(-\infty,\infty)$.
In the former (respectively latter) case, pull this strand to the left of $a-\ell$ (respectively $a$);
the resulting diagram factors through an idempotent with loading $\mathbf{I}_\nu^\theta$, for some $\nu$ such that $|L_a(\nu)|> |L_a(\la)|$ (respectively $|I_a(\nu)| + |L_a(\nu)| >|I_a(\la)| + |L_a(\la)|$), and so is zero in $A_{\Lambda_a(\la)}$.  

Now consider the element $\varphi(d) = d' \otimes d'' \in A(n_L,\theta,\kappa) \otimes A(n_R,\theta,\kappa)$.
In the former (respectively latter) case, there is a strand in the diagram $d' \in A(n_L,\theta,\kappa)$ (respectively $d'' \in A(n_R,\theta,\kappa)$) which can be pulled to the left of $a-\ell$ (respectively $a$);
the resulting diagram factors through an idempotent with loading $\mathbf{I}_{\nu}^\theta$, for some $\nu\in \mptn \ell {n_L}$ (respectively for some $\nu\in \mptn \ell {n_R}$) such that $|L_a(\nu)|> |L_a(\la^L)|$ (respectively $|I_a(\nu)| + |L_a(\nu)| >|I_a(\la^R)| + |L_a(\la^R)|$) and so is zero in $A_{\Lambda^L_a(\la^L)}$ (respectively $A_{\Lambda^R_a(\la^R)}$).
The result follows.
\end{proof}

\begin{cor}\label{main}
Suppose $R$ is a field, and let $(\la, \mu)$ be a pair of $\ell$-multipartitions of $n$ such that $\mu \domby_\theta \la$.
If, for some $a \in \RR$, $(\la,\mu)$ admits a $\theta$-diagonal cut at $x=a$, then
\[
d_{\la \mu}(t)= d_{\la^L \mu^L}(t) \times d_{\la^R \mu^R}(t)
\]
and, furthermore,
\[
\Ext^k_{A(n,\theta,\kappa)} (\Delta(\la), \Delta(\mu)) \cong
\bigoplus_{i+j=k}
\Ext^i_{A(n_L,\theta,\kappa)} (\Delta(\la^L), \Delta(\mu^L)) \otimes 
\Ext^j_{A(n_R,\theta,\kappa)} (\Delta(\la^R), \Delta(\mu^R)),
\] 
\[
\Ext^k_{A(n,\theta,\kappa)} (\Delta(\la), L(\mu)) \cong
\bigoplus_{i+j=k}
\Ext^i_{A(n_L,\theta,\kappa)} (\Delta(\la^L), L(\mu^L)) \otimes 
\Ext^j_{A(n_R,\theta,\kappa)} (\Delta(\la^R), L(\mu^R)),
\]   
where $n_L = |\la^L| = |\mu^L|$ and $n_R = |\la^R| = |\mu^R|$.
\end{cor}

\begin{proof}
This follows by \cref{3.6rmk,algisom}.
\end{proof}

\begin{rem}
As previously noted, in the level 1 case,  a pair $(\la,\mu)$ such that $\la \trianglerighteq \mu$ admits a $\theta$-diagonal cut at $x=a$ if and only if it admits a horizontal cut, if and only if it admits a vertical cut.  These horizontal and vertical cuts cross at the `highest' node in the diagonal ${\bf I}_\la^\theta \cap (a-\ell, a)$; that is, the node $(r,c)$ satisfying Equation~\ref{cutineq} for which $r+c$ is maximal.
In this case, the corresponding horizontal and vertical cuts are after the $r$th row and the $c$th column, respectively.
In this case, the ungraded version of the result above is a corollary of \cite[4.2(9)]{Donkin} and \cite[Section 10]{donkinhandbook}.
(We remark that the graded case is new, although the result for graded decomposition numbers was already known over $\CC$ by calculations in the  Fock space -- see \cite{cmt02}).

For (graded) decomposition numbers and higher extension groups of cyclotomic $q$-Schur algebras (in other words, the algebras $A(n,\theta,\kappa)$ for $\theta$ a well-separated weighting) the result above is completely new.
In this case it is easy to see that a pair of multipartitions $(\la,\mu)$ admits a $\theta$-diagonal cut at $x=a$ if and only if it admits a horizontal cut after the $r$th row  of the $m$th component   (in the sense of \cite{fs16}), where $(r,c,m)$ is the highest node in $(a-\ell, a)$, as in \cref{explicitpieces}.
Thus we retrieve an analogue of \cite[Theorem 4.8]{fs16} (which is stated for the cyclotomic Hecke algebras) for the cyclotomic $q$-Schur algebras.
Here, our `left piece' is the same as the `top piece' in \cite[Theorem 4.8]{fs16} (modulo empty components which do not play an important role where these decomposition numbers are concerned), while our right piece differs from their `bottom piece' in component $m$; we have an extra $r$ rows of length $c$, which are common to $\la^R$ and $\mu^R$.
In order to remove this `extra rectangle'  it would suffice to establish a `first row removal' isomorphism of algebras relating algebras of different degrees and {\em more importantly} different $e$-multicharges and weightings; this would proceed by analogy with \cite{fs16}.
Given an algebra  with a well-separated weighting, it is not difficult to establish such an isomorphism on the level of graded vector spaces and thus deduce the equality of graded decomposition numbers over fields of characteristic zero (by arguing in an LLT-type fashion as in \cite[Theorem 4.12]{bs15}).
However, the obvious map does not lift to the level of an isomorphism of algebras and so we do not include this here.
\end{rem}

\begin{eg}
Recall our running example, with $e=3$, $\theta=0$, $\la=(5,4,3,2,1)$, and $\mu=(4^3,1^3)$.
We have that 
\[
d_{(5,4,3,2,1)(4^3,1^3)}(t) = d_{(5,4,3)(4^3)}(t)d_{(3^3,2,1)(3^3,1^3)}(t) = t \times t= t^2.
\]
regardless of the   characteristic of the underlying field $R$.  
Similarly, 
\begin{align*}
	&\Dim{(\Hom_{A(15,(0),(0))}(\Delta(5,4,3,2,1),\Delta(4^3,1^3)))} \\
=	&\Dim{(\Hom_{A(12,(0),(0))}(\Delta(5,4,3),\Delta(4^3)))}  \times 
		\Dim{(\Hom_{A(12,(0),(0))}(\Delta(3^3,2,1), \Delta(3^3,1^3)))} \\
= &t \times t = t^2.
\end{align*}
\end{eg}

\begin{eg}
Let $e=5$, $\kappa=(0,2)$, and $\theta=(0,1)$.
The bipartitions
\[
\la = ((11,9,7,3^2,2,1^3),(9,4,2,1^4)) \quad \text{and} \quad \mu = ((10,9,8,4,3,1^5), (8,4,2,1^4))
\] 
admit a diagonal cut at $x = 5.2$, which yields
\begin{alignat*}2
\la^L &= ((11,9,7,3^2), (9,4,2)),
&\qquad
\la^R &= ((3^5,2,1^3),(1^7)),\\
\mu^L &= ((10,9,8,4,3),(8,4,2)),
&\qquad
\mu^R &= ((3^5,1^5),(1^7)).
\end{alignat*}

These bipartitions are depicted in \cref{level1,level2} in the introduction.
Applying \cref{main}, we have
\[
d_{\la\mu}(t) = d_{\la^L \mu^L}(t) \times d_{\la^R \mu^R}(t).
\]
First of all, we note that   \cite[Theorem 4.30]{bs15} is insufficient to calculate $d_{\la\mu} (t)$; this is because   $\la$ and $\mu$ differ in nodes of adjacent residues ($0$ and $1$).
However, \cite[Theorem 4.30]{bs15} is sufficient to calculate both $d_{\la^L\mu^L}(t)$ and  $d_{\la^R\mu^R}(t)$, because $\la^L$ and $\mu^L$ (respectively $\la^R$ and $\mu^R$) differ only in nodes of a single residue, namely 0 (respectively 1).
Applying \cite[Theorem 4.30]{bs15}, we deduce that
\[
d_{\la\mu}(t) = (t^5 + t^3) \times t^2 = t^7 + t^5.
\]
\end{eg}

\begin{eg}\label{cylotomic}
Let $e=3$, $\kappa = (0,1)$, and $\theta =(0,115)$ a well-separated weighting. The bipartitions $\la = ((5^2,4^2,3,2,1),(9,6,4^2,3,2^3,1))$ and $\mu = ((5,4^2,3^3),(9,6,5,4^2,2^2,1^3))$ of $57$ admit a $\theta$-diagonal cut 
at $x= 121$. This cut results in
\begin{alignat*}2
\la^L &= ((5^2,4^2,3,2,1),(9,6,4^2,3)),
&\qquad
\la^R &= (\varnothing,(3^5,2^3,1)),\\
\mu^L &= ((5,4^2,3^3),(9,6,5,4^2)),
&\qquad
\mu^R &= (\varnothing,(3^5,2^2,1^3)).
\end{alignat*}
This reduction again yields multipartitions amenable to the techniques of \cite{bs15} (whereas $\la$ and $\mu$ are not). 
The left-hand pieces should be compared with \cite[Example 2.6]{bs15}, which yields $d_{\la^L \mu^L} = t^{11} + 2 t^9 + 2 t^7 + t^5$. We may also apply \cite[Theorem 4.30]{bs15} to calculate that $d_{\la^R \mu^R} = t$. Thus, we have $d_{\la \mu} = t^{12} + 2 t^{10} + 2 t^8 + t^6$.
\end{eg}

\providecommand{\bysame}{\leavevmode\hbox to3em{\hrulefill}\thinspace}
\providecommand{\MR}{\relax\ifhmode\unskip\space\fi MR }
\providecommand{\MRhref}[2]{%
  \href{http://www.ams.org/mathscinet-getitem?mr=#1}{#2}
}
\providecommand{\href}[2]{#2}

\end{document}